\documentclass[reqno]{amsart}
\usepackage[latin1]{inputenc}
\usepackage[pdftex]{graphicx}
\usepackage{graphics}
\usepackage{graphicx}
\usepackage{epstopdf}
\usepackage[pdftex,hyperindex]{hyperref}
\usepackage{floatflt}
\usepackage{multicol}
\usepackage[T1]{fontenc}
\usepackage{textcomp}
\usepackage{latexsym}
\usepackage{expdlist}
\usepackage{vmargin}
\usepackage{booktabs}
\usepackage{placeins}
\usepackage{array}
\usepackage{amsmath}
\usepackage{amssymb}
\usepackage{amsfonts}
\usepackage{verbatim}
\usepackage{array}
\usepackage{framed}
\usepackage{amsthm}
\usepackage{enumerate}
\usepackage{cite}
\input amssym.def
\input amssym
\theoremstyle{plain}
\newtheorem{theorem}{Theorem}[section]
\newtheorem{lemma}{Lemma}[section]
\newtheorem{corollary}{Corollary}[section]
\newtheorem{proposition}{Proposition}[section]
\theoremstyle{remark}
\newtheorem{remark}{Remark}[section]

\theoremstyle{definition}

\allowdisplaybreaks[1]

\begin{document}

\title[Associated symmetric Pollaczek polynomials]{Tur\'{a}n's inequality, nonnegative linearization and amenability properties for associated symmetric Pollaczek polynomials}

\author{Stefan Kahler}

\address{Lehrstuhl A f\"{u}r Mathematik, RWTH Aachen University, 52056 Aachen, Germany}

\email{stefan.kahler@mathA.rwth-aachen.de}

\date{\today}

\begin{abstract}
An elegant and fruitful way to bring harmonic analysis into the theory of orthogonal polynomials and special functions, or to associate certain Banach algebras with orthogonal polynomials satisfying a specific but frequently satisfied nonnegative linearization property, is the concept of a polynomial hypergroup. Polynomial hypergroups (or the underlying polynomials, respectively) are accompanied by $L^1$-algebras and a rich, well-developed and unified harmonic analysis. However, the individual behavior strongly depends on the underlying polynomials. We study the associated symmetric Pollaczek polynomials, which are a two-parameter generalization of the ultraspherical polynomials. Considering the associated $L^1$-algebras, we will provide complete characterizations of weak amenability and point amenability by specifying the corresponding parameter regions. In particular, we shall see that there is a large parameter region for which none of these amenability properties holds (which is very different to $L^1$-algebras of locally compact groups). Moreover, we will rule out right character amenability. The crucial underlying nonnegative linearization property will be established, too, which particularly establishes a conjecture of R. Lasser (1994). Furthermore, we shall prove Tur\'{a}n's inequality for associated symmetric Pollaczek polynomials. Our strategy relies on chain sequences, asymptotic behavior, further Tur\'{a}n type inequalities and transformations into more convenient orthogonal polynomial systems.
\end{abstract}



\maketitle

\numberwithin{equation}{section}

\section{Introduction}\label{sec:intro}

\subsection{Motivation}

A classical and well-known result of Tur\'{a}n states that if $(\Lambda_n(x))_{n\in\mathbb{N}_0}$ is the sequence of Legendre polynomials, then
\begin{equation*}
(\Lambda_n(x))^2-\Lambda_{n+1}(x)\Lambda_{n-1}(x)\geq0
\end{equation*}
for all $n\in\mathbb{N}$ and $x\in[-1,1]$, and equality holds if and only if $x=\pm1$ \cite{Sz48,Tu50} (see \cite{BS09,Sz98} for a historic overview). A remarkable and much more nontrivial result of Gasper characterizes the parameters $\alpha,\beta>-1$ for which the Jacobi polynomials $(R_n^{(\alpha,\beta)}(x))_{n\in\mathbb{N}_0}$, which are orthogonal w.r.t. the measure $(1-x)^{\alpha}(1+x)^{\beta}\chi_{(-1,1)}(x)\,\mathrm{d}x$ and normalized by $R_n^{(\alpha,\beta)}(1)\equiv1$, satisfy an analogous inequality \cite{Ga72}. The special case of ultraspherical polynomials, which corresponds to $\alpha=\beta$ and also contains the originally considered Legendre polynomials ($\alpha=\beta=0$), can be tackled easier \cite{Sk54,Sz48,Sz51}. There are remarkable and more recent refinements of Tur\'{a}n's inequality for Legendre \cite{AGKL07} and ultraspherical \cite{NP15} polynomials (partially based on computer methods). General criteria in terms of the three-term recurrence relations of orthogonal polynomials were obtained in \cite{BS09,Kr11,Sz98}.\\

In this paper, we consider the class of associated symmetric Pollaczek polynomials, which is another generalization of the ultraspherical polynomials. In contrast to the Jacobi polynomials, the corresponding orthogonalization measures are still symmetric but depend on \textit{two} additional parameters $\lambda$ and $\nu$ of rather different origin and nature. The corresponding polynomials are a three-parameter generalization of the Legendre polynomials, and explicit computations are often not possible (except for special choices of the parameters). For instance, products of the form $R_m^{(\alpha,\beta)}(x)R_n^{(\alpha,\beta)}(x)$ can be expanded w.r.t. the basis $\{R_k^{(\alpha,\beta)}(x):k\in\mathbb{N}_0\}$ via ${}_9F_8$ hypergeometric series \cite{Ra81a}, and in another result Gasper characterizes all $\alpha,\beta>-1$ for which the arising linearization coefficients are nonnegative (which is called `property (P)' in the literature), using explicit recurrence relations \cite{Ga70a,Ga70b} or explicit linearization formulas for related classes \cite{Ga83}. We are not aware of analogous formulas for the (whole) class of associated symmetric Pollaczek polynomials, however.\\

One main objective of this paper is to establish both Tur\'{a}n's inequality and property (P) for associated symmetric Pollaczek polynomials. Inequalities for associated Pollaczek polynomials are a nontrivial and topical field; see \cite{LR20} for a very recent publication. Our proof of property (P) particularly shows a conjecture of Lasser \cite{La94} (in fact, we shall obtain a stronger result). Property (P) gives rise to a deep and rich harmonic analysis and, in particular, to a certain $L^1$-algebra \cite{La83}. Another main objective of the paper is to study amenability properties of this Banach algebra (depending on the three parameters $\alpha$, $\lambda$ and $\nu$), which will extend earlier results obtained in \cite{Ka15}.\footnote{Results which are cited from \cite{Ka15} can also be found in \cite{Ka16b}.} In doing so, we will also deal with the (not explicitly available) derivatives of the polynomials.

\subsection{Underlying setting and basics about polynomial hypergroups}

Let us consider the basic underlying setting in some more detail. Let $(P_n(x))_{n\in\mathbb{N}_0}\subseteq\mathbb{R}[x]$ with $\mathrm{deg}\;P_n(x)=n$ be a sequence of polynomials which satisfies a recurrence relation of the type $P_0(x)=1$, $P_1(x)=x$ and
\begin{equation*}
x P_n(x)=a_n P_{n+1}(x)+c_n P_{n-1}(x)\;(n\in\mathbb{N}),
\end{equation*}
where $(a_n)_{n\in\mathbb{N}}\subseteq(0,1)$ and $(c_n)_{n\in\mathbb{N}}\subseteq(0,1)$ fulfill $a_n+c_n=1\;(n\in\mathbb{N})$. Moreover, let $a_0:=1$ and $c_0:=0$ (so the recurrence is also satisfied for $n=0$). The sequence $(\sigma_n(x))_{n\in\mathbb{N}_0}\subseteq\mathbb{R}[x]$ of monic polynomials that corresponds to $(P_n(x))_{n\in\mathbb{N}_0}$ clearly satisfies the recurrence relation $\sigma_0(x)=1$, $\sigma_1(x)=x$,
\begin{equation*}
x\sigma_n(x)=\sigma_{n+1}(x)+c_n a_{n-1}\sigma_{n-1}(x)\;(n\in\mathbb{N}).
\end{equation*}
Let $(P_n(x))_{n\in\mathbb{N}_0}$ fulfill `property (P)', i.e., all linearization coefficients $g(m,n;k)$ given by the expansions
\begin{equation}\label{eq:productlinear}
P_m(x)P_n(x)=\sum_{k=0}^{m+n}\underbrace{g(m,n;k)}_{\overset{!}{\geq}0\;\mbox{(P)}}P_k(x)\;(m,n\in\mathbb{N}_0)
\end{equation}
are nonnegative. As a consequence of Favard's theorem and well-known uniqueness results from the theory of orthogonal polynomials \cite{Ch78}, $(P_n(x))_{n\in\mathbb{N}_0}$ is orthogonal w.r.t. a unique symmetric probability (Borel) measure $\mu$ on $\mathbb{R}$ with $|\mathrm{supp}\;\mu|=\infty$, i.e.,
\begin{equation*}
\int_\mathbb{R}\!P_m(x)P_n(x)\,\mathrm{d}\mu(x)\neq0\Leftrightarrow m=n.
\end{equation*}
It is also well-known that the zeros of the polynomials are real and located in the interior of the convex hull of $\mathrm{supp}\;\mu$ \cite{Ch78}; moreover, one obviously has the normalization $P_n(1)\equiv1$. Due to orthogonality, one has $g(m,n;|m-n|),g(m,n;m+n)\neq0$ and $g(m,n;k)=0$ for $k<|m-n|$ \cite{La05} (so the summation in \eqref{eq:productlinear} starts with $k=|m-n|$, in fact). Due to symmetry, one also has $g(m,n;k)=0$ if $m+n-k$ is odd. Another obvious consequence is that $\sum_{k=|m-n|}^{m+n}g(m,n;k)=1$. Defining a convolution which maps $\mathbb{N}_0\times\mathbb{N}_0$ into the convex hull of the Dirac functions on $\mathbb{N}_0$ via
\begin{equation}\label{eq:convfirst}
(m,n)\mapsto\sum_{k=|m-n|}^{m+n}g(m,n;k)\delta_k,
\end{equation}
and defining an involution $\mathbb{N}_0\rightarrow\mathbb{N}_0$ by the identity, one obtains a commutative discrete hypergroup with unit element $0$ on the nonnegative integers (induced by $(P_n(x))_{n\in\mathbb{N}_0}$).\footnote{As usual, $\delta$ with a subscript means the corresponding Dirac function.} Such hypergroups were introduced by Lasser \cite{La83} and are called polynomial hypergroups on $\mathbb{N}_0$. They are generally very different from groups or semigroups, and the individual behavior strongly depends on the underlying sequence $(P_n(x))_{n\in\mathbb{N}_0}$ (there is an abundance of examples). Nevertheless, many concepts of harmonic analysis take a rather unified and concrete form, which makes these objects located at a fruitful crossing point between the theory of orthogonal polynomials and special functions, on the one hand, and functional and harmonic analysis and the theory of Banach algebras, on the other hand.\\

To make the paper more self-contained (and since the reader is not expected to have preknowledge on hypergroups), we briefly recall some basics and, if not stated otherwise, refer to \cite{La83,La05} in the following. We previously mention that the hypergroup structure appears only ,,in the background''; in fact, we are mainly interested in the polynomials themselves and in the corresponding $L^1$-algebras (which can be directly described in terms of the linearization coefficients $g(m,n;k)$; see below).\\

Roughly speaking, the concept of a hypergroup generalizes (locally compact) groups by allowing the convolution of two Dirac measures to be not necessarily a Dirac measure again but a more general probability measure satisfying certain compatibility and non-degeneracy properties; the full axioms can be found in standard literature like \cite{BH95}, and we particularly refer to \cite{La05} for the simplified axioms concerning discrete hypergroups. In the following, let property (P) be satisfied.
\begin{itemize}
\item For any function $f:\mathbb{N}_0\rightarrow\mathbb{C}$ and any $n\in\mathbb{N}_0$, the translation $T_n f:\mathbb{N}_0\rightarrow\mathbb{C}$ of $f$ by $n$ is given by
\begin{equation*}
T_n f(m)=\sum_{k=|m-n|}^{m+n}g(m,n;k)f(k).
\end{equation*}
There is a Haar measure w.r.t. such translations (cf. equation \eqref{eq:haarbackground} below): normalized such that $\{0\}$ is mapped to $1$, it is just the counting measure on $\mathbb{N}_0$ weighted by the Haar weights, i.e., the values of the Haar function $h:\mathbb{N}_0\rightarrow[1,\infty)$,
\begin{equation*}
h(n):=\frac{1}{g(n,n;0)}=\frac{1}{\int_\mathbb{R}\!P_n^2(x)\,\mathrm{d}\mu(x)}.
\end{equation*}
Equivalently, $h$ is recursively given by
\begin{equation*}
h(0)=1,\;h(n+1)=\frac{a_n}{c_{n+1}}h(n)\;(n\in\mathbb{N}_0).
\end{equation*}
For every $p\in[1,\infty)$, let $\ell^p(h):=\{f:\mathbb{N}_0\rightarrow\mathbb{C}:\left\|f\right\|_p<\infty\}$ with $\left\|f\right\|_p:=\left(\sum_{k=0}^\infty|f(k)|^p h(k)\right)^{1/p}$. Moreover, let $\ell^\infty(h):=\ell^\infty$. Then the translation of an element of $\ell^p(h)$ is an element of $\ell^p(h)$ again (for every $p\in[1,\infty]$). One has
\begin{equation}\label{eq:haarbackground}
\sum_{k=0}^{\infty}T_n f(k)h(k)=\sum_{k=0}^{\infty}f(k)h(k)
\end{equation}
for every $f\in\ell^1(h)$ and $n\in\mathbb{N}_0$.
\item For $f\in\ell^p(h)$ and $g\in\ell^q(h)$, where $p\in[1,\infty]$ and $q:=p/(p-1)\in[1,\infty]$, the convolution $f\ast g:\mathbb{N}_0\rightarrow\mathbb{C}$ is defined by
\begin{equation}\label{eq:convsecond}
f\ast g(n):=\sum_{k=0}^\infty T_n f(k)g(k)h(k),
\end{equation}
and one has $f\ast g=g\ast f\in\ell^\infty$ \cite{Ka15,La05}. Moreover, if $f\in\ell^1(h)$, then $f\ast g\in\ell^q(h)$ with $\left\|f\ast g\right\|_q\leq\left\|f\right\|_1\left\|g\right\|_q$. Together with this convolution \eqref{eq:convsecond} (which is an extension of the hypergroup convolution \eqref{eq:convfirst}), complex conjugation and the norm $\left\|.\right\|_1$, the set $\ell^1(h)$ becomes a semisimple commutative Banach $\ast$-algebra with unit $\delta_0$ (the `$L^1$-algebra'), and $\ell^\infty$ is the dual module of $\ell^1(h)$ (acting via convolution) \cite{La07,La09b}.
\item Let
\begin{equation*}
\mathcal{X}^b(\mathbb{N}_0):=\left\{z\in\mathbb{C}:\sup_{n\in\mathbb{N}_0}|P_n(z)|<\infty\right\}=\left\{z\in\mathbb{C}:\max_{n\in\mathbb{N}_0}|P_n(z)|=1\right\}
\end{equation*}
(the latter equality is not obvious but always valid), and let $\widehat{\mathbb{N}_0}:=\mathcal{X}^b(\mathbb{N}_0)\cap\mathbb{R}$. $\mathcal{X}^b(\mathbb{N}_0)$ can be identified with the structure space $\Delta(\ell^1(h))$ via the homeomorphism $\mathcal{X}^b(\mathbb{N}_0)\rightarrow\Delta(\ell^1(h))$, $z\mapsto\varphi_z$, $\varphi_z(f):=\sum_{k=0}^\infty f(k)\overline{P_k(z)}h(k)\;(f\in\ell^1(h))$. In the same way, $\widehat{\mathbb{N}_0}$ can be identified with the Hermitian structure space $\Delta_s(\ell^1(h))$. In particular, Gelfand's theory yields that $\mathcal{X}^b(\mathbb{N}_0)$ and $\widehat{\mathbb{N}_0}$ are compact. Furthermore,
\begin{equation*}
\{1\}\cup\mathrm{supp}\;\mu\subseteq\widehat{\mathbb{N}_0}\subseteq[-1,1].
\end{equation*}
Given some $z\in\mathcal{X}^b(\mathbb{N}_0)$, the character $\alpha_z\in\ell^\infty\backslash\{0\}$ belonging to $z$ is given by
\begin{equation*}
\alpha_z(n):=P_n(z)\;(n\in\mathbb{N}_0).
\end{equation*}
One has
\begin{equation*}
T_m\alpha_z(n)=\alpha_z(m)\alpha_z(n)\;(m,n\in\mathbb{N}_0)
\end{equation*}
and, obviously,
\begin{equation*}
|\alpha_z(n)|\leq1\;(n\in\mathbb{N}_0).
\end{equation*}
If $x\in\widehat{\mathbb{N}_0}$, then $\alpha_x\in\ell^\infty\backslash\{0\}$ is called a symmetric character.
\item Finally, for $f\in\ell^1(h)$, the Fourier transform $\widehat{f}:\widehat{\mathbb{N}_0}\rightarrow\mathbb{C}$ reads
\begin{equation*}
\widehat{f}(x)=\sum_{k=0}^\infty f(k)P_k(x)h(k).
\end{equation*}
$\widehat{f}$ is continuous, and one has $\left\|\widehat{f}\right\|_\infty\leq\left\|f\right\|_1$ and $\widehat{f\ast g}=\widehat{f}\;\widehat{g}\;(f,g\in\ell^1(h))$. The Plancherel--Levitan theorem states that
\begin{equation*}
\left\|f\right\|_2^2=\left\|\widehat{f}\right\|_2^2=\int_\mathbb{R}\!(\widehat{f}(x))^2\,\mathrm{d}\mu(x)\;(f\in\ell^1(h)).
\end{equation*}
The Fourier transformation $(\ell^1(h),\left\|.\right\|_1)\rightarrow(\mathcal{C}(\widehat{\mathbb{N}_0}),\left\|.\right\|_{\infty})$, $f\mapsto\widehat{f}$ is injective and continuous. There is exactly one isometric isomorphism $\mathcal{P}:(\ell^2(h),\left\|.\right\|_2)\rightarrow(L^2(\mathbb{R},\mu),\left\|.\right\|_2)$, called the Plancherel isomorphism, such that $\widehat{f}=\mathcal{P}(f)$ for all $f\in\ell^1(h)$ (equality as elements of $L^2(\mathbb{R},\mu)$). The orthogonalization measure $\mu$ serves as `Plancherel measure'; one has
\begin{equation*}
\mathcal{P}^{-1}(F)(k)=\int_\mathbb{R}\!F(x)P_k(x)\,\mathrm{d}\mu(x)\;(F\in L^2(\mathbb{R},\mu),k\in\mathbb{N}_0).
\end{equation*}
\end{itemize}
The nonnegative linearization property (P) is crucial for the harmonic analysis described above.

\subsection{Associated symmetric Pollaczek polynomials and outline of the paper}

We are interested in the following class: let $\alpha>-1/2$, $\lambda\geq0$ and $\nu\geq0$. The sequence $(P_n(x))_{n\in\mathbb{N}_0}=:(Q_n^{(\alpha,\lambda,\nu)}(x))_{n\in\mathbb{N}_0}$ of associated symmetric Pollaczek polynomials which corresponds to $\alpha$, $\lambda$ and $\nu$ is given by
\begin{equation}\label{eq:cassocpoll}
c_n=\frac{n+\nu+2\alpha}{2n+2\nu+2\alpha+2\lambda+1}\frac{L_{n-1}^{(2\alpha,\nu)}(-2\lambda)}{L_n^{(2\alpha,\nu)}(-2\lambda)}\in(0,1)\;(n\in\mathbb{N})
\end{equation}
and
\begin{equation*}
a_n\equiv1-c_n,
\end{equation*}
where $(L_n^{(2\alpha,\nu)}(x))_{n\in\mathbb{N}_0}$ denotes the associated Laguerre polynomials that correspond to $2\alpha$ and $\nu$. The latter are given by the recurrence relation $L_0^{(2\alpha,\nu)}(x)=1$, $L_1^{(2\alpha,\nu)}(x)=(-x+2\nu+2\alpha+1)/(\nu+1)$,
\begin{equation}\label{eq:Laguerrerec}
(n+\nu+1)L_{n+1}^{(2\alpha,\nu)}(x)=(-x+2n+2\nu+2\alpha+1)L_n^{(2\alpha,\nu)}(x)-(n+\nu+2\alpha)L_{n-1}^{(2\alpha,\nu)}(x)\;(n\in\mathbb{N}).
\end{equation}
Concerning well-definedness in \eqref{eq:cassocpoll}, we note that the associated Laguerre polynomials $(L_n^{(2\alpha,\nu)}(x))_{n\in\mathbb{N}_0}$ are positive on $(-\infty,0]$ \cite{La94}. $(c_n)_{n\in\mathbb{N}}$ is also given via the recurrence relation
\begin{equation}\label{eq:assocpollrec}
c_n=\frac{\frac{(n+\nu)(n+\nu+2\alpha)}{(2n+2\nu+2\alpha+2\lambda+1)(2n+2\nu+2\alpha+2\lambda-1)}}{1-c_{n-1}}\;(n\in\mathbb{N}).
\end{equation}
These basics can be found in \cite{Ch78,La94}; an explicit formula for the orthogonalization measure was found in \cite{Po50} and reads $\mathrm{d}\mu(x)=\mu^{\prime}(x)\,\mathrm{d}x$ with
\begin{equation}\label{eq:muassocpoll}
\mu^{\prime}(x)=\begin{cases} C_{\alpha,\lambda,\nu}\frac{(1-x^2)^\alpha e^{\frac{\lambda x(2\arccos x-\pi)}{\sqrt{1-x^2}}}\left|\Gamma\left(\alpha+\nu+\frac{1}{2}+\frac{i\lambda x}{\sqrt{1-x^2}}\right)\right|^2}{\left|{}_2F_1\left(\left.\begin{matrix}\frac{1}{2}-\alpha+\frac{i\lambda x}{\sqrt{1-x^2}},\nu \\ \alpha+\nu+\frac{1}{2}+\frac{i\lambda x}{\sqrt{1-x^2}}\end{matrix}\right|2x^2-1+2i x\sqrt{1-x^2}\right)\right|^2}, & x\in(-1,1), \\ 0, & \mbox{else}, \end{cases}
\end{equation}
where $C_{\alpha,\lambda,\nu}>0$ is a constant such that $\mu$ has total mass $1$. There are also hypergeometric representations of the associated Pollaczek polynomials \cite{LW19,Wi90}.\\

In the theory of orthogonal polynomials and asymptotics, this class is of special interest because it is well-known that for certain choices of the parameters Szeg{\H{o}}'s condition is not satisfied (i.e., one does not have $\int_{(-1,1)}\!\log(\mu^{\prime}(x))/\sqrt{1-x^2}\,\mathrm{d}\lambda(x)>-\infty$ for certain choices of the parameters, which has consequences on asymptotic behavior) \cite{Sz75}. For $\lambda=0$, one gets the associated ultraspherical polynomials. For $\nu=0$, one gets the symmetric Pollaczek polynomials. Finally, for $\lambda=\nu=0$ one gets the ultraspherical polynomials, including the Legendre polynomials ($\alpha=0$) and the Chebyshev polynomials of the first ($\alpha=-1/2$) and second ($\alpha=1/2$) kind.\\

In \cite{La94}, property (P) was established for the case $\alpha\geq0$, for the case $\lambda=0$ and for the case $\nu=0\wedge\lambda<\alpha+1/2$; some partial results concerning property (P) were also obtained in \cite{La83}. Moreover, in \cite{La94} Lasser conjectured that property (P) is satisfied whenever $\lambda<\alpha+1/2$. In Section~\ref{sec:Turan}, we give a stronger result than the conjectured one and show that nonnegative linearization is \textit{always} satisfied for the class of associated symmetric Pollaczek polynomials (Theorem~\ref{thm:assocpollnonneg}). Moreover, we prove that Tur\'{a}n's inequality is always satisfied: one has
\begin{align*}
(Q_n^{(\alpha,\lambda,\nu)}(x))^2-Q_{n+1}^{(\alpha,\lambda,\nu)}(x)Q_{n-1}^{(\alpha,\lambda,\nu)}(x)&\geq0\;(n\in\mathbb{N},x\in[-1,1]),\\
(Q_n^{(\alpha,\lambda,\nu)}(x))^2-Q_{n+1}^{(\alpha,\lambda,\nu)}(x)Q_{n-1}^{(\alpha,\lambda,\nu)}(x)&>0\;(n\in\mathbb{N},x\in(-1,1)).
\end{align*}
Concerning the class of symmetric Pollaczek polynomials, i.e., concerning the special case $\nu=0$, Tur\'{a}n's inequality was already tackled in \cite[Theorem 3]{Sz98}. However, the proof of \cite[Theorem 3]{Sz98} contains an error which leaves open the case $\lambda<\alpha+1/2$.\footnote{The critical point in the proof given in \cite{Sz98} is the application of \cite[Corollary 1]{Sz98}. If $\lambda<\alpha+1/2$, then the conditions of this corollary are not satisfied (more precisely, both condition (i) and condition (ii) of the corollary are violated).} With a normalization other than at $x=1$, the special case $\nu=0$ was also studied in \cite{BI97}.\\

One of the most important topics in harmonic analysis is amenability and its various generalizations, historically motivated by the Banach--Tarski paradox and paradoxical decompositions \cite{Wa93}. In the context of (polynomial or different) hypergroups, recently studied amenability properties are generalizations of Reiter's condition \cite{Ku20} and F{\o}lner type conditions \cite{Al17}, for instance. In Section~\ref{sec:amenability}, we continue our research started in \cite{Ka15} and study amenability properties which correspond to the associated symmetric Pollaczek polynomials. In particular, we are interested in weak amenability of $\ell^1(h)$ (nonexistence of nonzero bounded derivations from $\ell^1(h)$ into $\ell^{\infty}$) and point amenability of $\ell^1(h)$ (nonexistence of nonzero bounded point derivations). The latter means that there is no $x\in\widehat{\mathbb{N}_0}$ which admits a nonzero bounded linear functional $D:\ell^1(h)\rightarrow\mathbb{C}$ satisfying
\begin{equation*}
D(f\ast g)=\widehat{f}(x)D(g)+\widehat{g}(x)D(f)
\end{equation*}
for all $f,g\in\ell^1(h)$.\footnote{In Section~\ref{sec:preliminaries}, amenability notions will be recalled in more detail (including references).} In contrast to the group case, such amenability properties are rather strong conditions on ($L^1$-algebras of) polynomial hypergroups. Both weak and point amenability can be characterized via properties of the derivatives of the polynomials, which makes such considerations also interesting with regard to special functions. Concerning the class of associated symmetric Pollaczek polynomials, we will provide full characterizations of weak amenability of $\ell^1(h)$, as well as full characterizations of point amenability of $\ell^1(h)$, by precisely specifying the corresponding parameter regions (Theorem~\ref{thm:pwamassocpoll}). Moreover, we shall consider further amenability properties and, for instance, rule out right character amenability.\\

Besides these central parts Section~\ref{sec:Turan} and Section~\ref{sec:amenability}, the paper is organized as follows: in Section~\ref{sec:preliminaries}, we collect some preliminaries and particularly recall a general criterion on Tur\'{a}n's inequality and property (P) due to Szwarc, some basics about chain sequences and, in greater detail, several amenability properties. Section~\ref{sec:rw} provides a useful transformation into systems with more convenient recurrence relations and easier asymptotic behavior. These systems are easily seen to satisfy Tur\'{a}n type inequalities; such inequalities will also be a crucial tool throughout the paper. Finally, in Section~\ref{sec:extended} we shall extend our main results to a larger parameter region, including some $\alpha\leq-1/2$ (Theorem~\ref{thm:extended}).

\section{Preliminaries}\label{sec:preliminaries}

Given some arbitrary sequence $(P_n(x))_{n\in\mathbb{N}_0}$ which satisfies a recurrence relation as in Section~\ref{sec:intro}, it may be very hard to check whether property (P) is satisfied or not, and we are not aware of any simple and convenient characterization (for instance, in terms of the coefficients $(a_n)_{n\in\mathbb{N}}$ and $(c_n)_{n\in\mathbb{N}}$, or in terms of the orthogonalization measure $\mu$). In a series of papers, Szwarc et al. gave several criteria which can help to tackle such problems, extending earlier work of Askey \cite{As70}. The following sufficient criterion is from Szwarc's result \cite[Theorem 1 p. 960]{Sz92b}:

\begin{theorem}\label{thm:szwarcnonneg}
If $(c_n)_{n\in\mathbb{N}}$ is nondecreasing and bounded from above by $1/2$, then property (P) is satisfied.
\end{theorem}

In \cite{Sz92b}, Theorem~\ref{thm:szwarcnonneg} has been successfully applied to all ultraspherical polynomials for which property (P) is valid. However, for the class of ultraspherical polynomials property (P) was fully understood much earlier and in a much more explicit way due to Dougall's formula for the $g(m,n;k)$ \cite[Theorem 6.8.2]{AAR99}. Hence, it is more interesting to apply Theorem~\ref{thm:szwarcnonneg} to classes for which explicit formulas for the $g(m,n;k)$ are not available---like the class of associated symmetric Pollaczek polynomials studied in this paper, which is a two-parameter generalization of the ultraspherical polynomials.\\

Under the same conditions as in Theorem~\ref{thm:szwarcnonneg}, Szwarc found a criterion for the validity of Tur\'{a}n's inequality \cite[Theorem 1]{Sz98}:

\begin{theorem}\label{thm:szwarcturan}
If $(c_n)_{n\in\mathbb{N}}$ is nondecreasing and bounded from above by $1/2$, then the polynomials $(P_n(x))_{n\in\mathbb{N}_0}$ satisfy Tur\'{a}n's inequality, i.e.,
\begin{align*}
(P_n(x))^2-P_{n+1}(x)P_{n-1}(x)&\geq0\;(n\in\mathbb{N},x\in[-1,1]),\\
(P_n(x))^2-P_{n+1}(x)P_{n-1}(x)&>0\;(n\in\mathbb{N},x\in(-1,1)).
\end{align*}
\end{theorem}

Theorem~\ref{thm:szwarcnonneg} and Theorem~\ref{thm:szwarcturan} will be crucial tools for our study of the associated symmetric Pollaczek polynomials. We shall also need the following, which is a consequence of \cite[Theorem (2.2)]{La94}:

\begin{theorem}\label{thm:lassersupportorthmeasure}
If $(c_n)_{n\in\mathbb{N}}$ is nondecreasing and convergent to $c\in(0,1/2]$, then (property (P) is satisfied, cf. above, and)
\begin{equation*}
\mathrm{supp}\;\mu=[-2\sqrt{c(1-c)},2\sqrt{c(1-c)}].
\end{equation*}
Moreover, if $c=1/2$, then
\begin{equation*}
\mathcal{X}^b(\mathbb{N}_0)=\widehat{\mathbb{N}_0}=\mathrm{supp}\;\mu=[-1,1].
\end{equation*}
\end{theorem}

Another important tool will be chain sequences. These sequences have various applications in the theory of orthogonal polynomials \cite{Ch78} and, in particular, in the theory of polynomial hypergroups (cf. \cite{La94}, for instance). Let us recall some basics.
\begin{itemize}
\item A sequence $(\Lambda_n)_{n\in\mathbb{N}}\subseteq(0,1)$ is called a chain sequence if there is some $p_0\in[0,1)$ and a sequence $(p_n)_{n\in\mathbb{N}}\subseteq(0,1)$ such that $\Lambda_n\equiv p_n(1-p_{n-1})$; the sequence $(p_n)_{n\in\mathbb{N}_0}$ is called a parameter sequence for $(\Lambda_n)_{n\in\mathbb{N}}$. If $(p_n)_{n\in\mathbb{N}_0}$ and $(p_n^{\prime})_{n\in\mathbb{N}_0}$ are two parameter sequences for $(\Lambda_n)_{n\in\mathbb{N}}$ and $p_0<p_0^{\prime}$, then $p_n<p_n^{\prime}$ for all $n\in\mathbb{N}$ (and vice versa). A parameter sequence $(m_n)_{n\in\mathbb{N}_0}$ is called minimal if $m_n<p_n\;(n\in\mathbb{N}_0)$ for every other parameter sequence $(p_n)_{n\in\mathbb{N}_0}$, and a parameter sequence $(M_n)_{n\in\mathbb{N}_0}$ is called maximal if $p_n<M_n\;(n\in\mathbb{N}_0)$ for every other parameter sequence $(p_n)_{n\in\mathbb{N}_0}$. For every chain sequence $(\Lambda_n)_{n\in\mathbb{N}}$, the minimal parameter sequence $(m_n)_{n\in\mathbb{N}_0}$ and the maximal parameter sequence $(M_n)_{n\in\mathbb{N}_0}$ exist, and one has $m_0=0$.
\item If $(p_n)_{n\in\mathbb{N}_0}$ and $(p_n^{\prime})_{n\in\mathbb{N}_0}$ are two parameter sequences for $(\Lambda_n)_{n\in\mathbb{N}}$, then at least one of them is the maximal parameter sequence or $\lim_{n\to\infty}p_n/p_n^{\prime}=1$. If $(p_n)_{n\in\mathbb{N}_0}\neq(M_n)_{n\in\mathbb{N}_0}$ is a parameter sequence for $(\Lambda_n)_{n\in\mathbb{N}}$, then the infinite product $\prod_{n=1}^{\infty}m_n/p_n$ converges absolutely.
\item If $(\Lambda_n)_{n\in\mathbb{N}}$ is nondecreasing, then $(m_n)_{n\in\mathbb{N}_0}$ is strictly increasing and $(M_n)_{n\in\mathbb{N}_0}$ is nonincreasing. Finally, if $\Lambda_n>1/4$ for all $n\in\mathbb{N}$, then $(m_n)_{n\in\mathbb{N}_0}$ is strictly increasing.
\end{itemize}
These basics can either be found in \cite{Ch78,Wa48} or are obvious.\\

Section~\ref{sec:amenability} is devoted to amenability properties in a Banach algebraic sense. Let $A$ be a Banach algebra. Recall that
\begin{itemize}
\item a linear mapping $D$ from $A$ into a Banach $A$-bimodule\footnote{A Banach $A$-bimodule is a Banach space which is also an $A$-bimodule and acts continuously \cite{Da00}.} $X$ is called a derivation if $D(ab)=a\cdot D(b)+D(a)\cdot b\;(a,b\in A)$, an inner derivation if $D(a)=a\cdot x-x\cdot a\;(a\in A)$ for some $x\in X$, and a point derivation at $\varphi\in\Delta(A)$ (structure space) if $X=\mathbb{C}$ and $D(a b)=\varphi(a)D(b)+\varphi(b)D(a)\;(a,b\in A)$ \cite{Da00},
\item $A$ is called amenable if for every Banach $A$-bimodule $X$ every bounded derivation into the dual module $X^\ast$ is an inner derivation \cite{Jo72}, weakly amenable if every bounded derivation into $A^\ast$ is an inner derivation \cite{Jo88}, $\varphi$-amenable w.r.t. $\varphi\in\Delta(A)$ if for every Banach $A$-bimodule $X$ such that $a\cdot x=\varphi(a)x\;(a\in A,x\in X)$ every bounded derivation from $A$ into the dual module $X^\ast$ is an inner derivation \cite{KLP08a}, and right character amenable if $A$ is $\varphi$-amenable for every $\varphi\in\Delta(A)$ and $A$ has a bounded right approximate identity \cite{KLP08b,Mo08}.
\end{itemize}
If there exists a nonzero bounded point derivation at some $\varphi\in\Delta(A)$, then $A$ necessarily fails to be weakly amenable \cite[Theorem 2.8.63]{Da00} and $A$ is also not $\varphi$-amenable \cite[Remark 2.4]{KLP08a} (hence, $A$ is not right character amenable and particularly not amenable). Moreover, if $A$ is commutative, then weak amenability reduces to the property that there exists no nonzero bounded derivation from $A$ into $A^\ast$ \cite{BCD87}.\\

If $G$ is a locally compact group, then the group algebra $L^1(G)$ is amenable if and only if $G$ is amenable in the group sense \cite{Jo72}. Furthermore, $L^1(G)$ is right character amenable if and only if $G$ is amenable \cite{KLP08b}. However, $L^1(G)$ is always weakly amenable \cite{Jo91}; in particular, there are no nonzero bounded point derivations.\\

A direct generalization of amenability in the group sense (i.e., the existence of a left-invariant mean on $L^{\infty}(G)$) to polynomial hypergroups is possible but not very fruitful because this property would always be satisfied just due to commutativity of these hypergroups \cite[Example 3.3 (a)]{Sk92} (similarly, it is well-known that all Abelian locally compact groups are amenable). It is much more interesting to consider the $L^1$-algebra $\ell^1(h)$ instead. There exist several general results on amenability properties of $\ell^1(h)$ \cite{Ka15,Ka16b,La07,La09c,La09a,La09b,LP10,Pe11,Wo84}. For instance, $\ell^1(h)$ fails to be amenable whenever $h(n)\to\infty\;(n\to\infty)$ \cite[Theorem 3]{La07}. There are many cases in which nonzero bounded point derivations exist \cite{Ka15,La09a} (which, of course, is very different to the group case recalled above). If one identifies the Hermitian structure space $\Delta_s(\ell^1(h))$ with $\widehat{\mathbb{N}_0}$ as recalled in Section~\ref{sec:intro}, a point derivation w.r.t. some $\varphi_x\in\Delta_s(\ell^1(h))$, $x\in\widehat{\mathbb{N}_0}$, becomes a linear functional $D_x:\ell^1(h)\rightarrow\mathbb{C}$ which satisfies
\begin{equation*}
D_x(f\ast g)=\widehat{f}(x)D_x(g)+\widehat{g}(x)D_x(f)\;(f,g\in\ell^1(h))
\end{equation*}
\cite{La09a}. We call $\ell^1(h)$ `point amenable' if there is no $x\in\widehat{\mathbb{N}_0}$ which admits a nonzero bounded point derivation.\footnote{Observe that we do not consider point derivations w.r.t. $\varphi\in\Delta(\ell^1(h))\backslash\Delta_s(\ell^1(h))$; however, the class of associated symmetric Pollaczek polynomials which is studied in this paper satisfies $\Delta(\ell^1(h))\backslash\Delta_s(\ell^1(h))=\emptyset$ anyway. Note that `point amenability' in our sense must not be confused with `pointwise amenability' considered in \cite{DL10}.} Hence, if the Banach algebra $\ell^1(h)$ is not point amenable, then it is neither weakly nor right character amenable (and particularly not amenable). \cite[Theorem 1]{La09a} provides a characterization in terms of the derivatives of the polynomials (see Theorem~\ref{thm:lasserpoint} below).\\

For the rest of the section, it is always supposed that property (P) is satisfied.

\begin{theorem}\label{thm:lasserpoint}
Let $x\in\widehat{\mathbb{N}_0}$. Then the following are equivalent:
\begin{enumerate}[(i)]
\item $\{P_n^{\prime}(x):n\in\mathbb{N}_0\}$ is bounded.
\item There exists a nonzero bounded point derivation at $x$.
\end{enumerate}
\end{theorem}

Also the following criteria \cite[Proposition 2.1]{Ka15} will be useful for our purposes:

\begin{proposition}\label{prp:kahlerpoint}
The following hold:
\begin{enumerate}[(i)]
\item If
\begin{equation*}
c_n a_{n-1}\leq\frac{1}{4}
\end{equation*}
for all $n\in\mathbb{N}$, then $\widehat{\mathbb{N}_0}=[-1,1]$ and each $x\in(-1,1)$ admits a nonzero bounded point derivation.
\item If $\limsup_{n\to\infty}c_n<1/2$, then $0\in\widehat{\mathbb{N}_0}$ and $0$ admits a nonzero bounded point derivation.
\end{enumerate}
\end{proposition}

Defining $(\kappa_n)_{n\in\mathbb{N}_0}\subseteq c_{00}$ via $\kappa_0:=0$ and the expansions
\begin{equation*}
P_n^\prime(x)=\sum_{k=0}^{n-1}\kappa_n(k)P_k(x)h(k),\;\kappa_n(k):=0\;(k\geq n)
\end{equation*}
for $n\in\mathbb{N}$, or, equivalently, via
\begin{equation*}
\kappa_n=\mathcal{P}^{-1}(P_n^\prime)\;(n\in\mathbb{N}_0),
\end{equation*}
one can characterize weak amenability as follows \cite[Theorem 2]{La07} (or \cite[Theorem 2]{La09b}):

\begin{theorem}\label{thm:lasserweak}
$\ell^1(h)$ is weakly amenable if and only if $\{\left\|\kappa_n\ast\varphi\right\|_\infty:n\in\mathbb{N}_0\}$ is unbounded for all $\varphi\in\ell^\infty\backslash\{0\}$.
\end{theorem}

In contrast to the characterization provided by Theorem~\ref{thm:lasserweak}, the (weaker) notion of point amenability corresponds to unboundedness of $\{\left\|\kappa_n\ast\varphi\right\|_\infty:n\in\mathbb{N}_0\}$ for all symmetric characters, i.e., for all $\varphi\in\{\alpha_x:x\in\widehat{\mathbb{N}_0}\}$ \cite{Ka15}. In the theory of orthogonal polynomials, the $\kappa_n$ are also of interest of their own and can be used for certain characterizations of ultraspherical polynomials \cite{Ka16a,LO08}, for instance.\footnote{Results which are cited from \cite{Ka16a} can be found in \cite{Ka16b}, too.}\\

Concerning Theorem~\ref{thm:lasserweak}, several problems occur: on the one hand, explicit linearizations of derivatives (i.e., the $\kappa_n$) and explicit linearizations of products (i.e., the $g(m,n;k)$ occurring in the convolution) are often out of reach. On the other hand, the characterization involves the whole space $\ell^\infty$---but many tools of harmonic analysis only work on proper subspaces. Based on Theorem~\ref{thm:lasserweak}, in \cite[Theorem 2.2]{Ka15} and \cite[Theorem 2.3]{Ka15} we found the following necessary criterion and sufficient criterion involving absolute continuity w.r.t. the Lebesgue--Borel measure on $\mathbb{R}$:

\begin{theorem}\label{thm:wamnec}
If $\ell^1(h)$ is weakly amenable, then $\mu$ has a singular part or the Radon--Nikodym derivative $\mu^\prime$ is not absolutely continuous (as a function) on $[\min\mathrm{supp}\;\mu,\max\mathrm{supp}\;\mu]$.
\end{theorem}

\begin{theorem}\label{thm:wamsuff}
If each of the conditions
\begin{enumerate}[(i)]
\item $\{\left\|\kappa_n\ast\varphi\right\|_\infty:n\in\mathbb{N}_0\}$ is unbounded for all $\varphi\in\ell^\infty\backslash\mathcal{O}(n^{-1})$,
\item $\mu$ is absolutely continuous, $\mathrm{supp}\;\mu=[-1,1]$, $\mu^\prime>0$ a.e. in $[-1,1]$,
\item $h(n)=\mathcal{O}(n^\alpha)$ (as $n\to\infty$) for some $\alpha\in[0,1)$,
\item $\sup_{n\in\mathbb{N}_0}\int_\mathbb{R}\!h^2(n)P_n^4(x)\,\mathrm{d}\mu(x)<\infty$
\end{enumerate}
holds, then $\ell^1(h)$ is weakly amenable.
\end{theorem}

Theorem~\ref{thm:wamnec}, Theorem~\ref{thm:wamsuff} and some further ingredients enabled us to completely characterize weak amenability for the important classes of Jacobi, symmetric Pollaczek and associated ultraspherical polynomials (which share the ultraspherical polynomials as common subclass) by precisely specifying the corresponding parameter regions \cite{Ka15}. Moreover, in \cite{Ka15} we obtained analogous characterizations for point amenability. Also the situation w.r.t. amenability, $\varphi$-amenability and right character amenability is completely clarified for these classes (cf. \cite{Ka15,Ka16b}). Our strategy for the remaining associated symmetric Pollaczek polynomials (i.e., for those which do neither belong to the symmetric Pollaczek nor to the associated ultraspherical subcase) will be quite different from Theorem~\ref{thm:lasserweak}, Theorem~\ref{thm:wamnec} and Theorem~\ref{thm:wamsuff}. We will use Theorem~\ref{thm:lasserpoint}, chain sequences, asymptotic behavior, appropriate transformations and Tur\'{a}n's inequality for a more convenient system of orthogonal polynomials. The latter will be introduced in the next section.

\section{Transformation into random walk polynomials}\label{sec:rw}

The proofs of our main results Theorem~\ref{thm:assocpollnonneg} and Theorem~\ref{thm:pwamassocpoll} below will be done via systems whose recurrence relations and asymptotic behavior are more accessible. For $a>1$, $b>0$ and $\nu\geq0$, let 
\begin{align*} \widetilde{c_n^s}&:=\begin{cases} 0, & n=0, \\ \frac{n+\nu}{(a+1)(n+\nu)+b}, & n\in\mathbb{N}, \end{cases}\\
\widetilde{a_n^s}&\equiv1-\widetilde{c_n^s},\\
c_n^s&:=\begin{cases} 0, & n=0, \\ \widetilde{c_1^s}\frac{a\nu+b}{(a+1)\nu+b}, & n=1, \\ \frac{\widetilde{c_n^s}\widetilde{a_{n-1}^s}}{1-c_{n-1}^s}, & \mbox{else}, \end{cases}\\
a_n^s&\equiv1-c_n^s.
\end{align*}
Note that $(c_n^s)_{n\in\mathbb{N}_0}$ is well-defined, that
\begin{equation*}
(c_n^s)_{n\in\mathbb{N}}\subseteq\left(0,\frac{1}{a+1}\right)
\end{equation*}
and that
\begin{equation*}
c_n^s\leq\widetilde{c_n^s}\;(n\in\mathbb{N}_0);
\end{equation*}
this can be seen as follows: since the cases $n=0$ and $n=1$ are clear, let $n\in\mathbb{N}$ be arbitrary but fixed and assume that $c_n^s$ is well-defined, $0<c_n^s<1/(a+1)$ and $c_n^s\leq\widetilde{c_n^s}$. Then $c_{n+1}^s=\widetilde{c_{n+1}^s}\widetilde{a_n^s}/(1-c_n^s)$ is well-defined, $c_{n+1}^s>0$ and
\begin{equation*}
c_{n+1}^s\leq\frac{\widetilde{c_{n+1}^s}\widetilde{a_n^s}}{1-\widetilde{c_n^s}}=\widetilde{c_{n+1}^s}<\frac{1}{a+1}.
\end{equation*}
Now let $(\widetilde{S_n^{(a,b,\nu)}}(x))_{n\in\mathbb{N}_0}\subseteq\mathbb{R}[x]$ and $(S_n^{(a,b,\nu)}(x))_{n\in\mathbb{N}_0}\subseteq\mathbb{R}[x]$ be defined via $\widetilde{S_0^{(a,b,\nu)}}(x)=S_0^{(a,b,\nu)}(x)=1$, $\widetilde{S_1^{(a,b,\nu)}}(x)=S_1^{(a,b,\nu)}(x)=x$,
\begin{align*}
x\widetilde{S_n^{(a,b,\nu)}}(x)&=\widetilde{a_n^s}\widetilde{S_{n+1}^{(a,b,\nu)}}(x)+\widetilde{c_n^s}\widetilde{S_{n-1}^{(a,b,\nu)}}(x)\;(n\in\mathbb{N}),\\
x S_n^{(a,b,\nu)}(x)&=a_n^s S_{n+1}^{(a,b,\nu)}(x)+c_n^s S_{n-1}^{(a,b,\nu)}(x)\;(n\in\mathbb{N}).
\end{align*}
Clearly, we have $\widetilde{S_n^{(a,b,\nu)}}(1)=S_n^{(a,b,\nu)}(1)=1$ for all $n\in\mathbb{N}_0$.\\

The polynomials $(\widetilde{S_n^{(a,b,0)}}(x))_{n\in\mathbb{N}_0}=(S_n^{(a,b,0)}(x))_{n\in\mathbb{N}_0}$ are the `random walk polynomials' considered in \cite{AI84,Ka15,La94}: for $\alpha>-1/2$ and $0<\lambda<\alpha+1/2$, one has
\begin{equation}\label{eq:assocpollrelationpre}
Q_n^{(\alpha,\lambda,0)}(x)=\frac{S_n^{\left(\frac{2\alpha+2\lambda+1}{2\alpha-2\lambda+1},(2\alpha+1)\frac{2\alpha+2\lambda+1}{2\alpha-2\lambda+1},0\right)}\left(\sqrt{1-\left(\frac{2\lambda}{2\alpha+1}\right)^2}x\right)}{S_n^{\left(\frac{2\alpha+2\lambda+1}{2\alpha-2\lambda+1},(2\alpha+1)\frac{2\alpha+2\lambda+1}{2\alpha-2\lambda+1},0\right)}\left(\sqrt{1-\left(\frac{2\lambda}{2\alpha+1}\right)^2}\right)}
\end{equation}
for all $n\in\mathbb{N}_0$, and the arising denominators are positive.\\

Obviously, the sequence $(\widetilde{c_n^s})_{n\in\mathbb{N}}$ is strictly increasing. We show that also $(c_n^s)_{n\in\mathbb{N}}$ is strictly increasing: for every $n\in\mathbb{N}$, we have
\begin{equation*}
c_{n+1}^s=\frac{\widetilde{c_{n+1}^s}\widetilde{a_n^s}}{1-c_n^s}>\frac{\widetilde{c_n^s}\widetilde{a_n^s}}{1-c_n^s}=\frac{\widetilde{c_n^s}(1-\widetilde{c_n^s})}{c_n^s(1-c_n^s)}c_n^s\geq c_n^s,
\end{equation*}
where the latter inequality follows because
\begin{equation*}
c_n^s\leq\widetilde{c_n^s}<\frac{1}{a+1}<\frac{1}{2}.
\end{equation*}
Since $(\widetilde{c_n^s})_{n\in\mathbb{N}},(c_n^s)_{n\in\mathbb{N}}\subseteq(0,1/2)$, Szwarc's criterion Theorem~\ref{thm:szwarcnonneg} now implies that both $(\widetilde{S_n^{(a,b,\nu)}}(x))_{n\in\mathbb{N}_0}$ and $(S_n^{(a,b,\nu)}(x))_{n\in\mathbb{N}_0}$ induce polynomial hypergroups on $\mathbb{N}_0$.\\

The corresponding monic versions $(\widetilde{\sigma_n^{(a,b,\nu)}}(x))_{n\in\mathbb{N}_0}\subseteq\mathbb{R}[x]$ and $(\sigma_n^{(a,b,\nu)}(x))_{n\in\mathbb{N}_0}\subseteq\mathbb{R}[x]$ are given by $\widetilde{\sigma_0^{(a,b,\nu)}}(x)=\sigma_0^{(a,b,\nu)}(x)=1$, $\widetilde{\sigma_1^{(a,b,\nu)}}(x)=\sigma_1^{(a,b,\nu)}(x)=x$,
\begin{align*}
x\widetilde{\sigma_n^{(a,b,\nu)}}(x)&=\widetilde{\sigma_{n+1}^{(a,b,\nu)}}(x)+\widetilde{\lambda_n}\widetilde{\sigma_{n-1}^{(a,b,\nu)}}(x)\;(n\in\mathbb{N}),\\
x\sigma_n^{(a,b,\nu)}(x)&=\sigma_{n+1}^{(a,b,\nu)}(x)+\lambda_n\sigma_{n-1}^{(a,b,\nu)}(x)\;(n\in\mathbb{N}),
\end{align*}
where $\widetilde{\lambda_n}\equiv\widetilde{c_n^s}\widetilde{a_{n-1}^s}$, $\lambda_n\equiv c_n^s a_{n-1}^s$ and consequently, by construction, $(\widetilde{\lambda_n})_{n\geq2}=(\lambda_n)_{n\geq2}$. Furthermore, observe that there is some $N\geq2$ such that
\begin{align*}
&\widetilde{\lambda_{n+1}}-\widetilde{\lambda_n}=\\
&=\lambda_{n+1}-\lambda_n=\\
&=\frac{(a-1)b(n+\nu)-a b+b^2-b}{((n+\nu-1)(a+1)+b)((n+\nu)(a+1)+b)((n+\nu+1)(a+1)+b)}>0\;(n\geq N);
\end{align*}
hence, the sequence $(\widetilde{\lambda_n})_{n\geq N}=(\lambda_n)_{n\geq N}$ is strictly increasing.\\

Let $\widetilde{\mu^s}$ and $\mu^s$ denote the corresponding orthogonalization measures, and let $\omega\in(0,1)$ be defined by
\begin{equation}\label{eq:omegadef}
\omega:=\frac{2\sqrt{a}}{a+1}.
\end{equation}

\begin{lemma}\label{lma:assocrworthmeasures}
Let $a>1$, $b>0$ and $\nu\geq0$. Then the orthogonalization measures satisfy
\begin{equation*}
\mathrm{supp}\;\widetilde{\mu^s}=\mathrm{supp}\;\mu^s=[-\omega,\omega].
\end{equation*}
\end{lemma}

\begin{proof}
Since obviously $\lim_{n\to\infty}\widetilde{c_n^s}=1/(a+1)$, we obtain that $\mathrm{supp}\;\widetilde{\mu^s}=[-\omega,\omega]$ from Theorem~\ref{thm:lassersupportorthmeasure}. Let $N\geq2$ be as above. Since the sequence $(\widetilde{\lambda_n})_{n\geq N}=(\lambda_n)_{n\geq N}$ is strictly increasing, and since both $(\widetilde{c_n^s})_{n\geq N-1}$ and $(c_n^s)_{n\geq N-1}$ are strictly increasing, neither $(\widetilde{c_n^s})_{n\geq N-1}$ nor $(c_n^s)_{n\geq N-1}$ is the maximal parameter sequence for $(\widetilde{\lambda_n})_{n\geq N}=(\lambda_n)_{n\geq N}$, so $\lim_{n\to\infty}c_n^s/\widetilde{c_n^s}=1$. The latter yields $\lim_{n\to\infty}c_n^s=1/(a+1)$, and we can apply Theorem~\ref{thm:lassersupportorthmeasure} again to obtain that also $\mathrm{supp}\;\mu^s=[-\omega,\omega]$.
\end{proof}
As an immediate consequence of Lemma~\ref{lma:assocrworthmeasures}, we obtain that both $\widetilde{S_n^{(a,b,\nu)}}(\omega)>0$ and $S_n^{(a,b,\nu)}(\omega)>0$ for all $n\in\mathbb{N}_0$, as well as $\widetilde{\sigma_n^{(a,b,\nu)}}(\omega)>0$ and $\sigma_n^{(a,b,\nu)}(\omega)>0$ for all $n\in\mathbb{N}_0$.\\

Our next lemma provides a relation to the associated symmetric Pollaczek polynomials. It is a direct generalization of \eqref{eq:assocpollrelationpre}. We will use that the monic versions $(q_n^{(\alpha,\lambda,\nu)}(x))_{n\in\mathbb{N}_0}\subseteq\mathbb{R}[x]$ which correspond to $(Q_n^{(\alpha,\lambda,\nu)}(x))_{n\in\mathbb{N}_0}$ ($\alpha>-1/2$, $\lambda\geq0$, $\nu\geq0$) are given by $q_0^{(\alpha,\lambda,\nu)}(x)=1$, $q_1^{(\alpha,\lambda,\nu)}(x)=x$ and
\begin{equation}\label{eq:assocpollmonic}
\begin{split}
x q_n^{(\alpha,\lambda,\nu)}(x)&=q_{n+1}^{(\alpha,\lambda,\nu)}(x)\\
&\quad+\frac{(n+\nu)(n+\nu+2\alpha)}{(2n+2\nu+2\alpha+2\lambda+1)(2n+2\nu+2\alpha+2\lambda-1)}q_{n-1}^{(\alpha,\lambda,\nu)}(x)\;(n\in\mathbb{N}),
\end{split}
\end{equation}
see \cite{La94} or \eqref{eq:assocpollrec}.

\begin{lemma}\label{lma:assocpollrelation}
Let $\alpha>-1/2$, $0<\lambda<\alpha+1/2$ and $\nu\geq0$. Then
\begin{equation*}
Q_n^{(\alpha,\lambda,\nu)}(x)=\frac{S_n^{\left(\frac{2\alpha+2\lambda+1}{2\alpha-2\lambda+1},(2\alpha+1)\frac{2\alpha+2\lambda+1}{2\alpha-2\lambda+1},\nu\right)}\left(\sqrt{1-\left(\frac{2\lambda}{2\alpha+1}\right)^2}x\right)}{S_n^{\left(\frac{2\alpha+2\lambda+1}{2\alpha-2\lambda+1},(2\alpha+1)\frac{2\alpha+2\lambda+1}{2\alpha-2\lambda+1},\nu\right)}\left(\sqrt{1-\left(\frac{2\lambda}{2\alpha+1}\right)^2}\right)}
\end{equation*}
for all $n\in\mathbb{N}_0$, and the arising denominators are positive.
\end{lemma}

\begin{proof}
Let
\begin{equation*}
a:=\frac{2\alpha+2\lambda+1}{2\alpha-2\lambda+1},\;b:=(2\alpha+1)\frac{2\alpha+2\lambda+1}{2\alpha-2\lambda+1}.
\end{equation*}
Then
\begin{equation*}
\omega=\sqrt{1-\left(\frac{2\lambda}{2\alpha+1}\right)^2},
\end{equation*}
and $(\sigma_n^{(a,b,\nu)}(x))_{n\in\mathbb{N}_0}$ satisfies $\sigma_0^{(a,b,\nu)}(x)=1$, $\sigma_1^{(a,b,\nu)}(x)=x$,
\begin{equation*}
x\sigma_n^{(a,b,\nu)}(x)=\sigma_{n+1}^{(a,b,\nu)}(x)+\lambda_n\sigma_{n-1}^{(a,b,\nu)}(x)\;(n\in\mathbb{N})
\end{equation*}
with
\begin{align*}
\lambda_1&=\frac{1+\nu}{\left(\frac{2\alpha+2\lambda+1}{2\alpha-2\lambda+1}+1\right)(1+\nu)+(2\alpha+1)\frac{2\alpha+2\lambda+1}{2\alpha-2\lambda+1}}\frac{\frac{2\alpha+2\lambda+1}{2\alpha-2\lambda+1}\nu+(2\alpha+1)\frac{2\alpha+2\lambda+1}{2\alpha-2\lambda+1}}{\left(\frac{2\alpha+2\lambda+1}{2\alpha-2\lambda+1}+1\right)\nu+(2\alpha+1)\frac{2\alpha+2\lambda+1}{2\alpha-2\lambda+1}}=\\
&=\omega^2\frac{(1+\nu)(1+\nu+2\alpha)}{(3+2\nu+2\alpha+2\lambda)(1+2\nu+2\alpha+2\lambda)}
\end{align*}
and, for all $n\geq2$,
\begin{align*}
\lambda_n&=\frac{n+\nu}{\left(\frac{2\alpha+2\lambda+1}{2\alpha-2\lambda+1}+1\right)(n+\nu)+(2\alpha+1)\frac{2\alpha+2\lambda+1}{2\alpha-2\lambda+1}}\\
&\quad\times\frac{\frac{2\alpha+2\lambda+1}{2\alpha-2\lambda+1}(n+\nu-1)+(2\alpha+1)\frac{2\alpha+2\lambda+1}{2\alpha-2\lambda+1}}{\left(\frac{2\alpha+2\lambda+1}{2\alpha-2\lambda+1}+1\right)(n+\nu-1)+(2\alpha+1)\frac{2\alpha+2\lambda+1}{2\alpha-2\lambda+1}}=\\
&=\omega^2\frac{(n+\nu)(n+\nu+2\alpha)}{(2n+2\nu+2\alpha+2\lambda+1)(2n+2\nu+2\alpha+2\lambda-1)}.
\end{align*}
In conclusion, we have
\begin{equation*}
x\sigma_n^{(a,b,\nu)}(x)=\sigma_{n+1}^{(a,b,\nu)}(x)+\omega^2\frac{(n+\nu)(n+\nu+2\alpha)}{(2n+2\nu+2\alpha+2\lambda+1)(2n+2\nu+2\alpha+2\lambda-1)}\sigma_{n-1}^{(a,b,\nu)}(x)
\end{equation*}
and consequently
\begin{equation*}
x\frac{\sigma_n^{(a,b,\nu)}(\omega x)}{\omega^n}=\frac{\sigma_{n+1}^{(a,b,\nu)}(\omega x)}{\omega^{n+1}}+\frac{(n+\nu)(n+\nu+2\alpha)}{(2n+2\nu+2\alpha+2\lambda+1)(2n+2\nu+2\alpha+2\lambda-1)}\frac{\sigma_{n-1}^{(a,b,\nu)}(\omega x)}{\omega^{n-1}}
\end{equation*}
for all $n\in\mathbb{N}$. Therefore, we obtain
\begin{equation*}
q_n^{(\alpha,\lambda,\nu)}(x)=\frac{\sigma_n^{(a,b,\nu)}(\omega x)}{\omega^n}
\end{equation*}
as a consequence of \eqref{eq:assocpollmonic}. This implies the assertion.
\end{proof}

\section{Monotonicity of the recurrence coefficients, nonnegative linearization and Tur\'{a}n's inequality}\label{sec:Turan}

The following theorem is our first main result.

\begin{theorem}\label{thm:assocpollnonneg}
Let $\alpha>-1/2$, $\lambda\geq0$ and $\nu\geq0$ be arbitrary, and let $P_n(x)=Q_n^{(\alpha,\lambda,\nu)}(x)\;(n\in\mathbb{N}_0)$. Then $(c_n)_{n\in\mathbb{N}}$ is strictly increasing, and $(P_n(x))_{n\in\mathbb{N}_0}$ satisfies both the nonnegative linearization property (P) and Tur\'{a}n's inequality.
\end{theorem}

As a consequence of Theorem~\ref{thm:assocpollnonneg}, we will obtain a bound for the recurrence coefficients in terms of the function $\phi:[1,\infty)\rightarrow(0,1)$,
\begin{equation}\label{eq:phidef}
\phi(x):=\frac{(x+\nu)(x+\nu+2\alpha)}{(2x+2\nu+2\alpha+2\lambda+1)(2x+2\nu+2\alpha+2\lambda-1)}
\end{equation}
($\alpha>-1/2$, $\lambda,\nu\geq0$). With this notation, \eqref{eq:assocpollrec} reads
\begin{equation}\label{eq:assocpollrecmod}
c_n=\frac{\phi(n)}{1-c_{n-1}}\;(n\in\mathbb{N}).
\end{equation}
As a consequence of \eqref{eq:assocpollrecmod} and $(c_n)_{n\in\mathbb{N}}\subseteq(0,1)$, we see that $\phi$ indeed maps into the interval $(0,1)$; alternatively, one can see this by rewriting \eqref{eq:phidef} as
\begin{equation*}
\phi(x)=1-\frac{(2\alpha+1)(3x+3\nu+2\alpha+4\lambda-1)+4\lambda(2x+2\nu+\lambda-1)+(3x+3\nu)(x+\nu-1)}{(2x+2\nu+2\alpha+2\lambda+1)(2x+2\nu+2\alpha+2\lambda-1)}.
\end{equation*}

\begin{corollary}\label{cor:assocpollnonneg}
Let $\alpha>-1/2$, $\lambda\geq0$ and $\nu\geq0$, and let $P_n(x)=Q_n^{(\alpha,\lambda,\nu)}(x)\;(n\in\mathbb{N}_0)$. Then
\begin{equation*}
c_n<\frac{1}{2}(1-\sqrt{\max\{0,1-4\phi(n+1)\}})
\end{equation*}
for all $n\in\mathbb{N}$.
\end{corollary}

Further bounds will be obtained in Lemma~\ref{lma:cest}.\\

Theorem~\ref{thm:assocpollnonneg} shows that $(Q_n^{(\alpha,\lambda,\nu)}(x))_{n\in\mathbb{N}_0}$ always induces a polynomial hypergroup on $\mathbb{N}_0$. The Haar weights are of subexponential growth and given by
\begin{equation*}
h(n)=\frac{(2n+2\nu+2\alpha+2\lambda+1)(\nu+1)_n}{(2\alpha+2\lambda+2\nu+1)(2\alpha+\nu+1)_n}(L_n^{(2\alpha,\nu)}(-2\lambda))^2\;(n\in\mathbb{N}_0),
\end{equation*}
cf. \cite{La94} ($(.)_n$ denotes the Pochhammer symbol). As a consequence of Theorem~\ref{thm:lassersupportorthmeasure}, we have $\mathcal{X}^b(\mathbb{N}_0)=\widehat{\mathbb{N}_0}=\mathrm{supp}\;\mu=[-1,1]$: as soon as one knows that $(c_n)_{n\in\mathbb{N}}$ is strictly increasing (or at least nondecreasing), then, due to \eqref{eq:assocpollrecmod} and the limiting behavior of $\phi$, it is clear that $(c_n)_{n\in\mathbb{N}}$ converges to $1/2$, so Theorem~\ref{thm:lassersupportorthmeasure} can be applied. We also observe that, in view of these considerations, the conditions of Szwarc's criteria Theorem~\ref{thm:szwarcnonneg} and Theorem~\ref{thm:szwarcturan} are verified as soon as one knows that $(c_n)_{n\in\mathbb{N}}$ is strictly increasing. We will also need the derivative of $\phi$, which, for all $x\in[1,\infty)$, is given by
\begin{equation*}
\phi^{\prime}(x)=\begin{cases} \frac{(2\alpha+1)(2\alpha-1)(2x+2\nu+2\alpha)}{(2x+2\nu+2\alpha+1)^2(2x+2\nu+2\alpha-1)^2}, & \lambda=0, \\ \frac{\eta(x)}{8\lambda(2x+2\nu+2\alpha+2\lambda+1)^2(2x+2\nu+2\alpha+2\lambda-1)^2}, & \lambda>0, \end{cases}
\end{equation*}
where we define $\eta:\mathbb{R}\rightarrow\mathbb{R}$,
\begin{equation}\label{eq:etadef}
\eta(x)=(8\lambda(x+\nu)+(2\alpha+2\lambda)^2-1)^2-(1-(2\alpha-2\lambda)^2)(1-(2\alpha+2\lambda)^2).
\end{equation}
Concerning Theorem~\ref{thm:assocpollnonneg}, we preliminarily note that the proof below will show that if $\lambda>0$ and $\lambda\geq-|\alpha|+1/2$, then $\phi$ is nondecreasing (so $(c_n)_{n\in\mathbb{N}_0}$ is strictly increasing as minimal parameter sequence of $(\phi(n))_{n\in\mathbb{N}}$). However, if $0<\lambda<-|\alpha|+1/2$, then the behavior of $\phi$ can be much less convenient. Instead, it would be a natural try to proceed as follows: the representation \eqref{eq:cassocpoll} gives
\begin{align*}
\frac{c_{n+1}}{c_n}&=\frac{\frac{n+\nu+2\alpha+1}{2n+2\nu+2\alpha+2\lambda+3}\frac{L_n^{(2\alpha,\nu)}(-2\lambda)}{L_{n+1}^{(2\alpha,\nu)}(-2\lambda)}}{\frac{n+\nu+2\alpha}{2n+2\nu+2\alpha+2\lambda+1}\frac{L_{n-1}^{(2\alpha,\nu)}(-2\lambda)}{L_n^{(2\alpha,\nu)}(-2\lambda)}}=\\
&=\frac{(n+\nu+2\alpha+1)(2n+2\nu+2\alpha+2\lambda+1)}{(n+\nu+2\alpha)(2n+2\nu+2\alpha+2\lambda+3)}\frac{(L_n^{(2\alpha,\nu)}(-2\lambda))^2}{L_{n+1}^{(2\alpha,\nu)}(-2\lambda)L_{n-1}^{(2\alpha,\nu)}(-2\lambda)}
\end{align*}
for all $n\in\mathbb{N}$. Now if $0<\lambda<-|\alpha|+1/2$ (and consequently $-2\alpha+2\lambda+1>0$), one has
\begin{equation*}
\frac{(n+\nu+2\alpha+1)(2n+2\nu+2\alpha+2\lambda+1)}{(n+\nu+2\alpha)(2n+2\nu+2\alpha+2\lambda+3)}=1+\frac{-2\alpha+2\lambda+1}{(n+\nu+2\alpha)(2n+2\nu+2\alpha+2\lambda+3)}>1.
\end{equation*}
Therefore, if there was a Tur\'{a}n type inequality for the associated Laguerre polynomials $(L_n^{(2\alpha,\nu)}(x))_{n\in\mathbb{N}_0}$, then we would easily obtain that $(c_n)_{n\in\mathbb{N}}$ is strictly increasing. However, one has
\begin{equation*}
\left(L_1^{\left(-\frac{1}{4},0\right)}\left(-\frac{1}{4}\right)\right)^2-L_2^{\left(-\frac{1}{4},0\right)}\left(-\frac{1}{4}\right)L_0^{\left(-\frac{1}{4},0\right)}\left(-\frac{1}{4}\right)=-\frac{1}{8}<0,
\end{equation*}
which shows that the approach does not work for $(\alpha,\lambda,\nu)=(-1/8,1/8,0)$, for instance. Another natural idea would be to use a more appropriate normalization and, more precisely, to modify the approach in such a way that one obtains differences of the form
\begin{equation*}
\left(\frac{L_n^{(2\alpha,\nu)}(-2\lambda)}{L_n^{(2\alpha,\nu)}(0)}\right)^2-\frac{L_{n+1}^{(2\alpha,\nu)}(-2\lambda)}{L_{n+1}^{(2\alpha,\nu)}(0)}\frac{L_{n-1}^{(2\alpha,\nu)}(-2\lambda)}{L_{n-1}^{(2\alpha,\nu)}(0)}.
\end{equation*}
In fact, this modified approach works for the special case $\nu=0$ (symmetric Pollaczek polynomials):\footnote{For better reading, in the following we use an additional superscript ``$(\nu=0)$'' when referring to the special case $\nu=0$.} it is immediate from \eqref{eq:Laguerrerec} that
\begin{equation}\label{eq:Laguerrezeroquot}
\frac{L_n^{(2\alpha,0)}(0)}{L_{n-1}^{(2\alpha,0)}(0)}=\frac{n+2\alpha}{n}\;(n\in\mathbb{N}),
\end{equation}
and \eqref{eq:cassocpoll} and \eqref{eq:Laguerrezeroquot} imply that
\begin{equation*}
\frac{c_{n+1}^{(\nu=0)}}{c_n^{(\nu=0)}}=\frac{\frac{n+2\alpha+1}{2n+2\alpha+2\lambda+3}\frac{L_n^{(2\alpha,0)}(-2\lambda)}{L_{n+1}^{(2\alpha,0)}(-2\lambda)}}{\frac{n+2\alpha}{2n+2\alpha+2\lambda+1}\frac{L_{n-1}^{(2\alpha,0)}(-2\lambda)}{L_n^{(2\alpha,0)}(-2\lambda)}}=\frac{(n+1)(2n+2\alpha+2\lambda+1)}{n(2n+2\alpha+2\lambda+3)}\frac{\left(\frac{L_n^{(2\alpha,0)}(-2\lambda)}{L_n^{(2\alpha,0)}(0)}\right)^2}{\frac{L_{n+1}^{(2\alpha,0)}(-2\lambda)}{L_{n+1}^{(2\alpha,0)}(0)}\frac{L_{n-1}^{(2\alpha,0)}(-2\lambda)}{L_{n-1}^{(2\alpha,0)}(0)}}
\end{equation*}
for all $n\in\mathbb{N}$. Now Tur\'{a}n's inequality for the (renormalized) Laguerre polynomials $\left(\frac{L_n^{(2\alpha,0)}(x)}{L_n^{(2\alpha,0)}(0)}\right)_{n\in\mathbb{N}_0}$ \cite{MN51,Sz98,TN51} states that
\begin{align}
\label{eq:LaguerreTuranfirst} \left(\frac{L_n^{(2\alpha,0)}(x)}{L_n^{(2\alpha,0)}(0)}\right)^2-\frac{L_{n+1}^{(2\alpha,0)}(x)}{L_{n+1}^{(2\alpha,0)}(0)}\frac{L_{n-1}^{(2\alpha,0)}(x)}{L_{n-1}^{(2\alpha,0)}(0)}&\geq0\;(x\in\mathbb{R}),\\
\label{eq:LaguerreTuransecond} \left(\frac{L_n^{(2\alpha,0)}(x)}{L_n^{(2\alpha,0)}(0)}\right)^2-\frac{L_{n+1}^{(2\alpha,0)}(x)}{L_{n+1}^{(2\alpha,0)}(0)}\frac{L_{n-1}^{(2\alpha,0)}(x)}{L_{n-1}^{(2\alpha,0)}(0)}&>0\;(x\in\mathbb{R}\backslash\{0\})
\end{align}
for all $n\in\mathbb{N}$. Hence, we have
\begin{equation}\label{eq:estimationforappendix}
\frac{c_{n+1}^{(\nu=0)}}{c_n^{(\nu=0)}}\begin{cases} >\frac{(n+1)(2n+2\alpha+2\lambda+1)}{n(2n+2\alpha+2\lambda+3)}, & \lambda>0, \\ =\frac{(n+1)(2n+2\alpha+1)}{n(2n+2\alpha+3)}, & \lambda=0. \end{cases}
\end{equation}
Since
\begin{equation*}
\frac{(n+1)(2n+2\alpha+2\lambda+1)}{n(2n+2\alpha+2\lambda+3)}=1+\frac{2\alpha+2\lambda+1}{n(2n+2\alpha+2\lambda+3)}>1,
\end{equation*}
we obtain that $(c_n^{(\nu=0)})_{n\in\mathbb{N}}$ is strictly increasing (observe that this argument works for all $\alpha>-1/2$, $\lambda\geq0$).\\

Concerning the question whether the preceding argument generalizes to the full case $0<\lambda<-|\alpha|+1/2$ (with arbitrary $\nu\geq0$), we mention that there is an analogue to equation \eqref{eq:Laguerrezeroquot}: for every $n\in\mathbb{N}$, one has
\begin{equation*}
\frac{L_n^{(2\alpha,\nu)}(0)}{L_{n-1}^{(2\alpha,\nu)}(0)}=\begin{cases} \frac{1}{n+\nu}\frac{(2\alpha+\nu)_{n+1}-(\nu)_{n+1}}{(2\alpha+\nu)_n-(\nu)_n}, & \alpha\neq0, \\ \frac{\sum_{k=0}^n\frac{1}{k+\nu}}{\sum_{k=0}^{n-1}\frac{1}{k+\nu}}, & \alpha=0,\;\nu>0, \\ 1, & \alpha=\nu=0 \end{cases}
\end{equation*}
\cite{AW84}. However, a generalization fails already for the reason that Tur\'{a}n's inequality \eqref{eq:LaguerreTuranfirst}, \eqref{eq:LaguerreTuransecond} does not generalize: for instance, one has
\begin{equation*}
\left(\frac{L_1^{(0,1+\sqrt{2})}\left(-\frac{1}{2}\right)}{L_1^{(0,1+\sqrt{2})}(0)}\right)^2-\frac{L_2^{(0,1+\sqrt{2})}\left(-\frac{1}{2}\right)}{L_2^{(0,1+\sqrt{2})}(0)}\frac{L_0^{(0,1+\sqrt{2})}\left(-\frac{1}{2}\right)}{L_0^{(0,1+\sqrt{2})}(0)}=-\frac{1}{(3+2\sqrt{2})^2}<0.
\end{equation*}
Nevertheless, we found two very different ways how the full case $0<\lambda<-|\alpha|+1/2$ can be successfully tackled via Tur\'{a}n type inequalities which are valid for suitable related classes of orthogonal polynomials:
\begin{itemize}
\item The faster way will be based on Theorem~\ref{thm:szwarcturan}, Lemma~\ref{lma:assocpollrelation} and Tur\'{a}n's inequality for the sequence $(S_n^{(a,b,\nu)}(x))_{n\in\mathbb{N}_0}$ (see Section~\ref{sec:rw}) with
\begin{equation*}
a:=\frac{2\alpha+2\lambda+1}{2\alpha-2\lambda+1},\;b:=(2\alpha+1)\frac{2\alpha+2\lambda+1}{2\alpha-2\lambda+1}.
\end{equation*}
It avoids an analysis of $\phi$ for the case $0<\lambda<-|\alpha|+1/2$, and it works in the larger region $0<\lambda<\alpha+1/2$. Moreover, if one is just interested in property (P) (and not in the monotonicity of the recurrence coefficients $(c_n)_{n\in\mathbb{N}}$, Tur\'{a}n's inequality for the associated symmetric Pollaczek polynomials and Corollary~\ref{cor:assocpollnonneg}), then Lemma~\ref{lma:assocpollrelation} provides a solution without any use of Tur\'{a}n type inequalities. The details will be given in the proof of Theorem~\ref{thm:assocpollnonneg} below.
\item The second way makes use of the special case $\nu=0$ and equation \eqref{eq:estimationforappendix} above (which was obtained via Tur\'{a}n's inequality for Laguerre polynomials \eqref{eq:LaguerreTuranfirst}, \eqref{eq:LaguerreTuransecond}). In contrast to the first, shorter way, it considers the behavior of $\phi$ for the case $0<\lambda<-|\alpha|+1/2$. Justified by the more classical character, and since it is of interest how the problem can be solved via Tur\'{a}n type inequalities in two very different ways, this second way will be presented in an appendix. Furthermore, in Lemma~\ref{lma:cest} we will obtain some estimations for the recurrence coefficients which may be helpful for other problems.
\end{itemize}

\begin{proof}[Proof (Theorem~\ref{thm:assocpollnonneg})]
We only have to show that $(c_n)_{n\in\mathbb{N}}$ is strictly increasing, which implies property (P) and Tur\'{a}n's inequality as consequences of Szwarc's criteria Theorem~\ref{thm:szwarcnonneg} and Theorem~\ref{thm:szwarcturan} (cf. above). If $\lambda=0$ and $\alpha\geq1/2$, then $\phi^{\prime}(x)\geq0$ for all $x\in[1,\infty)$, so $\phi$ is nondecreasing. If $\lambda=0$ and $\alpha<1/2$, then
\begin{equation*}
\phi(x)-\frac{1}{4}=\frac{(1-2\alpha)(1+2\alpha)}{4(2x+2\nu+2\alpha+1)(2x+2\nu+2\alpha-1)}>0
\end{equation*}
for all $x\in[1,\infty)$. Hence, in both cases we get that $(c_n)_{n\in\mathbb{N}}$ is strictly increasing by applying the theory of chain sequences as recalled in Section~\ref{sec:preliminaries}.\\

From now on, let $\lambda>0$.\\

Let $\lambda\geq|\alpha|+1/2$. If $\alpha\geq0$, then $\lambda\geq1/2$ and consequently $2\alpha+2\lambda\geq1$, so
\begin{align*}
\eta(x)&\geq(8\lambda+(2\alpha+2\lambda)^2-1)^2-(1-(2\alpha-2\lambda)^2)(1-(2\alpha+2\lambda)^2)=\\
&=16\lambda((4\alpha+4)\lambda^2+(8\alpha^2+8\alpha+4)\lambda+4\alpha^3+4\alpha^2-\alpha-1)\geq\\
&\geq16\lambda\left((4\alpha+4)\left(\alpha+\frac{1}{2}\right)^2+(8\alpha^2+8\alpha+4)\left(\alpha+\frac{1}{2}\right)+4\alpha^3+4\alpha^2-\alpha-1\right)=\\
&=32(2\alpha+1)^3\lambda>\\
&>0
\end{align*}
for all $x\in[1,\infty)$, so $\phi$ is strictly increasing. If $\alpha<0$, then $2\alpha+2\lambda\geq1$ again and we obtain
\begin{align*}
\eta(x)&\geq(8\lambda+(2\alpha+2\lambda)^2-1)^2-(1-(2\alpha-2\lambda)^2)(1-(2\alpha+2\lambda)^2)=\\
&=16\lambda((4\alpha+4)\lambda^2+(8\alpha^2+8\alpha+4)\lambda+4\alpha^3+4\alpha^2-\alpha-1)\geq\\
&\geq16\lambda\left((4\alpha+4)\left(-\alpha+\frac{1}{2}\right)^2+(8\alpha^2+8\alpha+4)\left(-\alpha+\frac{1}{2}\right)+4\alpha^3+4\alpha^2-\alpha-1\right)=\\
&=32(1-2\alpha)\lambda>\\
&>0
\end{align*}
for all $x\in[1,\infty)$, so $\phi$ is strictly increasing, too.\\

Now let $-|\alpha|+1/2<\lambda<|\alpha|+1/2$. If $\alpha\geq0$, then $2\alpha+2\lambda>1$ and consequently we obtain as above
\begin{align*}
\eta(x)&\geq(8\lambda+(2\alpha+2\lambda)^2-1)^2-(1-(2\alpha-2\lambda)^2)(1-(2\alpha+2\lambda)^2)=\\
&=16\lambda((4\alpha+4)\lambda^2+(8\alpha^2+8\alpha+4)\lambda+4\alpha^3+4\alpha^2-\alpha-1)
\end{align*}
for all $x\in[1,\infty)$. On the one hand, this implies
\begin{align*}
\eta(x)&>16\lambda\left((4\alpha+4)\left(-\alpha+\frac{1}{2}\right)^2+(8\alpha^2+8\alpha+4)\left(-\alpha+\frac{1}{2}\right)+4\alpha^3+4\alpha^2-\alpha-1\right)=\\
&=32(1-2\alpha)\lambda
\end{align*}
for all $x\in[1,\infty)$ if $\alpha\leq1/2$; on the other hand, we get
\begin{equation*}
\eta(x)>16\lambda(4\alpha^3+4\alpha^2-\alpha-1)=16(\alpha+1)(2\alpha+1)(2\alpha-1)\lambda
\end{equation*}
for all $x\in[1,\infty)$. Putting both together, we can conclude that $\eta(x)>0$ for all $x\in[1,\infty)$. Thus, $\phi$ is strictly increasing. If $\alpha<0$, however, then
\begin{equation*}
(1-(2\alpha-2\lambda)^2)(1-(2\alpha+2\lambda)^2)=((1+2\alpha)^2-4\lambda^2)((1-2\alpha)^2-4\lambda^2)<0,
\end{equation*}
so we obtain that $\eta(x)>0$ for all $x\in[1,\infty)$, too, and $\phi$ is strictly increasing again.\\

Let $\lambda=-|\alpha|+1/2$. Then, for all $x\in[1,\infty)$,
\begin{equation*}
\eta(x)=\begin{cases} 64\lambda^2(x+\nu)^2, & \alpha\geq0, \\ 64\lambda^2(x+\nu+2\alpha)^2, & \alpha<0. \end{cases}
\end{equation*}
Therefore, $\eta(x)>0$ for all $x\in[1,\infty)$, so $\phi$ is strictly increasing also in this case.\\

However, if the chain sequence $(\phi(n))_{n\in\mathbb{N}}$ is strictly increasing (or at least nondecreasing), then its minimal parameter sequence $(c_n)_{n\in\mathbb{N}_0}$ is strictly increasing, cf. above. Hence, it remains to consider the case $\lambda<-|\alpha|+1/2$. We just assume that $\lambda<\alpha+1/2$ in the following, and we make use of the random walk polynomials considered in Section~\ref{sec:rw}.\\

Let
\begin{equation*}
a:=\frac{2\alpha+2\lambda+1}{2\alpha-2\lambda+1},\;b:=(2\alpha+1)\frac{2\alpha+2\lambda+1}{2\alpha-2\lambda+1}
\end{equation*}
and
\begin{equation*}
S_n(x):=S_n^{(a,b,\nu)}(x)\;(n\in\mathbb{N}_0).
\end{equation*}
Then
\begin{equation*}
\omega=\sqrt{1-\left(\frac{2\lambda}{2\alpha+1}\right)^2}
\end{equation*}
\eqref{eq:omegadef}. Since $(S_n(x))_{n\in\mathbb{N}_0}$ satisfies property (P) (due to Szwarc's criterion Theorem~\ref{thm:szwarcnonneg}, cf. Section~\ref{sec:rw}), Lemma~\ref{lma:assocpollrelation} implies that $(P_n(x))_{n\in\mathbb{N}_0}$ satisfies property (P), too---this argument is a generalization of Lasser's proof for the special case $(Q_n^{(\alpha,\lambda,0)}(x))_{n\in\mathbb{N}_0}$ \cite{La94}. Note that at this stage we have particularly established Lasser's conjecture on property (P). Concerning the full assertion of Theorem~\ref{thm:assocpollnonneg}, it is left to show that $(c_n)_{n\in\mathbb{N}}$ is strictly increasing again (hence, also in the case $\lambda<\alpha+1/2$ Szwarc's criteria Theorem~\ref{thm:szwarcnonneg} and Theorem~\ref{thm:szwarcturan} can be directly applied to $(P_n(x))_{n\in\mathbb{N}_0}$): Lemma~\ref{lma:assocpollrelation} yields
\begin{equation*}
P_n(x)=\frac{S_n(\omega x)}{S_n(\omega)}\;(n\in\mathbb{N}_0),
\end{equation*}
so
\begin{equation*}
x\frac{S_n(\omega x)}{S_n(\omega)}=a_n\frac{S_{n+1}(\omega x)}{S_{n+1}(\omega)}+c_n\frac{S_{n-1}(\omega x)}{S_{n-1}(\omega)}
\end{equation*}
and consequently
\begin{equation*}
x S_n(x)=\omega a_n\frac{S_n(\omega)}{S_{n+1}(\omega)}S_{n+1}(x)+\omega c_n\frac{S_n(\omega)}{S_{n-1}(\omega)}S_{n-1}(x)
\end{equation*}
and finally
\begin{equation*}
c_n^s=\omega c_n\frac{S_n(\omega)}{S_{n-1}(\omega)}
\end{equation*}
for all $n\in\mathbb{N}$. Since $(c_n^s)_{n\in\mathbb{N}}\subseteq(0,1/2)$ is strictly increasing (see Section~\ref{sec:rw}), Theorem~\ref{thm:szwarcturan} implies that $(S_n(x))_{n\in\mathbb{N}_0}$ satisfies Tur\'{a}n's inequality, i.e.,
\begin{align*}
(S_n(x))^2-S_{n+1}(x)S_{n-1}(x)&\geq0\;(x\in[-1,1]),\\
(S_n(x))^2-S_{n+1}(x)S_{n-1}(x)&>0\;(x\in(-1,1))
\end{align*}
for all $n\in\mathbb{N}$. Together with the monotonicity of $(c_n^s)_{n\in\mathbb{N}}$, we can conclude that
\begin{align*}
c_{n+1}-c_n&=\frac{c_{n+1}^s S_n(\omega)}{\omega S_{n+1}(\omega)}-\frac{c_n^s S_{n-1}(\omega)}{\omega S_n(\omega)}=\\
&=\frac{c_{n+1}^s(S_n(\omega))^2-c_n^s S_{n+1}(\omega)S_{n-1}(\omega)}{\omega S_{n+1}(\omega)S_n(\omega)}>\\
&>\frac{c_n^s(S_n(\omega))^2-c_n^s S_{n+1}(\omega)S_{n-1}(\omega)}{\omega S_{n+1}(\omega)S_n(\omega)}>\\
&>0
\end{align*}
for all $n\in\mathbb{N}$.
\end{proof}

\begin{proof}[Proof (Corollary~\ref{cor:assocpollnonneg})]
This follows from Theorem~\ref{thm:assocpollnonneg} and the identity
\begin{equation*}
(c_{n+1}-c_n)(1-c_n)=\left(c_n-\frac{1}{2}\right)^2+\phi(n+1)-\frac{1}{4}\;(n\in\mathbb{N}).
\end{equation*}
\end{proof}

\section{Amenability properties}\label{sec:amenability}

Concerning amenability properties, we show the following:

\begin{theorem}\label{thm:pwamassocpoll}
Let $\alpha>-1/2$, $\lambda\geq0$, $\nu\geq0$ and $P_n(x)=Q_n^{(\alpha,\lambda,\nu)}(x)\;(n\in\mathbb{N}_0)$. Then $\ell^1(h)$ is
\begin{enumerate}[(i)]
\item point amenable if and only if $\alpha<1/2$ and $\lambda=0$,
\item weakly amenable if and only if $\alpha<0$ and $\lambda=\nu=0$,
\item never right character amenable,
\item never amenable.
\end{enumerate}
\end{theorem}

We use the notation of the previous sections. The following lemma is needed for the proof of Theorem~\ref{thm:pwamassocpoll} and provides an asymptotic relation between the sequences $(S_n^{(a,b,\nu)}(\omega))_{n\in\mathbb{N}_0}$ and $(\widetilde{S_n^{(a,b,\nu)}}(\omega))_{n\in\mathbb{N}_0}$ (cf. Section~\ref{sec:rw}). Again, a Tur\'{a}n type inequality will play a crucial role.

\begin{lemma}\label{lma:assymptoticrelation}
Let $a>1$, $b>0$ and $\nu\geq0$. Then there is some $\tau>0$ such that
\begin{equation*}
\frac{S_n^{(a,b,\nu)}(\omega)}{\widetilde{S_n^{(a,b,\nu)}}(\omega)}\to\tau\;(n\to\infty).
\end{equation*}
\end{lemma}

\begin{proof}
In the following, let $N\geq2$ be as in Section~\ref{sec:rw} (so $(\widetilde{\lambda_n})_{n\geq N}=(\lambda_n)_{n\geq N}$ is strictly increasing). The proof will be divided into three steps and use the monic versions $(\sigma_n^{(a,b,\nu)}(x))_{n\in\mathbb{N}_0}$, $(\widetilde{\sigma_n^{(a,b,\nu)}}(x))_{n\in\mathbb{N}_0}$:\\

\textit{Step 1:} we show that there is some $\tau_1>0$ such that
\begin{equation*}
\frac{\sigma_n^{(a,b,\nu)}(1)}{\widetilde{\sigma_n^{(a,b,\nu)}}(1)}\to\tau_1\;(n\to\infty).
\end{equation*}
Comparing the leading coefficients of $\widetilde{S_n^{(a,b,\nu)}}(x)$ and $\widetilde{\sigma_n^{(a,b,\nu)}}(x)$, we see that $\widetilde{\sigma_n^{(a,b,\nu)}}(1)=\prod_{k=0}^{n-1}\widetilde{a_k^s}$ for every $n\in\mathbb{N}$; in the same way, one has $\sigma_n^{(a,b,\nu)}(1)=\prod_{k=0}^{n-1}a_k^s$. Consequently, we have
\begin{equation*}
\frac{\sigma_n^{(a,b,\nu)}(1)}{\widetilde{\sigma_n^{(a,b,\nu)}}(1)}=\frac{\prod_{k=0}^{n-1}a_k^s}{\prod_{k=0}^{n-1}\widetilde{a_k^s}}=\frac{\prod_{k=1}^{n-1}a_k^s}{\prod_{k=1}^{n-1}\widetilde{a_k^s}}=\frac{\prod_{k=1}^{n-1}\frac{\widetilde{\lambda_{k+1}}}{\widetilde{a_k^s}}}{\prod_{k=1}^{n-1}\frac{\lambda_{k+1}}{a_k^s}}=\frac{\prod_{k=2}^n\widetilde{c_k^s}}{\prod_{k=2}^n c_k^s}
\end{equation*}
for each $n\geq2$. Since neither $(\widetilde{c_n^s})_{n\geq N-1}$ nor $(c_n^s)_{n\geq N-1}$ is the maximal parameter sequence for $(\widetilde{\lambda_n})_{n\geq N}=(\lambda_n)_{n\geq N}$ (cf. the proof of Lemma~\ref{lma:assocrworthmeasures}), the infinite products $\prod_{n=N}^{\infty}m_n/\widetilde{c_n^s}$ and $\prod_{n=N}^{\infty}m_n/c_n^s$ converge absolutely, where $(m_n)_{n\geq N-1}$ shall denote the minimal parameter sequence for $(\widetilde{\lambda_n})_{n\geq N}=(\lambda_n)_{n\geq N}$. This establishes the assertion.\\

\textit{Step 2:} we show that there is some $\tau_{\omega}>0$ such that
\begin{equation*}
\frac{\sigma_n^{(a,b,\nu)}(\omega)}{\widetilde{\sigma_n^{(a,b,\nu)}}(\omega)}\to\tau_{\omega}\;(n\to\infty).
\end{equation*}
Let $(\widetilde{\chi_n})_{n\in\mathbb{N}}\subseteq(0,1)$ and $(\chi_n)_{n\in\mathbb{N}}\subseteq(0,1)$ be defined by
\begin{align*}
\widetilde{\chi_n}&:=1-\frac{\widetilde{\sigma_{n+1}^{(a,b,\nu)}}(\omega)}{\omega\widetilde{\sigma_n^{(a,b,\nu)}}(\omega)}=\frac{\widetilde{\lambda_n}\widetilde{\sigma_{n-1}^{(a,b,\nu)}}(\omega)}{\omega\widetilde{\sigma_n^{(a,b,\nu)}}(\omega)},\\
\chi_n&:=1-\frac{\sigma_{n+1}^{(a,b,\nu)}(\omega)}{\omega\sigma_n^{(a,b,\nu)}(\omega)}=\frac{\lambda_n\sigma_{n-1}^{(a,b,\nu)}(\omega)}{\omega\sigma_n^{(a,b,\nu)}(\omega)}.
\end{align*}
If $n\geq2$, then
\begin{equation*}
\widetilde{\chi_n}(1-\widetilde{\chi_{n-1}})=\frac{\widetilde{\lambda_n}\widetilde{\sigma_{n-1}^{(a,b,\nu)}}(\omega)}{\omega\widetilde{\sigma_n^{(a,b,\nu)}}(\omega)}\frac{\widetilde{\sigma_n^{(a,b,\nu)}}(\omega)}{\omega\widetilde{\sigma_{n-1}^{(a,b,\nu)}}(\omega)}=\frac{\widetilde{\lambda_n}}{\omega^2},
\end{equation*}
and in the same way we have
\begin{equation*}
\chi_n(1-\chi_{n-1})=\frac{\lambda_n}{\omega^2}=\frac{\widetilde{\lambda_n}}{\omega^2}.
\end{equation*}
Hence, $(\widetilde{\lambda_n}/\omega^2)_{n\geq2}=(\lambda_n/\omega^2)_{n\geq2}\subseteq(0,1)$ is a chain sequence and both $(\widetilde{\chi_n})_{n\in\mathbb{N}}$ and $(\chi_n)_{n\in\mathbb{N}}$ are parameter sequences. By construction, we have $\widetilde{\sigma_n^{(a,b,\nu)}}(\omega)=\omega^n\prod_{k=1}^{n-1}(1-\widetilde{\chi_k})$ and $\sigma_n^{(a,b,\nu)}(\omega)=\omega^n\prod_{k=1}^{n-1}(1-\chi_k)$ for every $n\geq2$. Therefore, we obtain
\begin{equation*}
\frac{\sigma_n^{(a,b,\nu)}(\omega)}{\widetilde{\sigma_n^{(a,b,\nu)}}(\omega)}=\frac{\prod_{k=1}^{n-1}(1-\chi_k)}{\prod_{k=1}^{n-1}(1-\widetilde{\chi_k})}=\frac{\prod_{k=1}^{n-1}\frac{\frac{\widetilde{\lambda_{k+1}}}{\omega^2}}{1-\widetilde{\chi_k}}}{\prod_{k=1}^{n-1}\frac{\frac{\lambda_{k+1}}{\omega^2}}{1-\chi_k}}=\frac{\prod_{k=2}^n\widetilde{\chi_k}}{\prod_{k=2}^n\chi_k}
\end{equation*}
for each $n\geq2$. We have
\begin{equation*}
\widetilde{\chi_1}=\frac{\widetilde{\lambda_1}\widetilde{\sigma_0^{(a,b,\nu)}}(\omega)}{\omega\widetilde{\sigma_1^{(a,b,\nu)}}(\omega)}=\frac{\widetilde{c_1^s}}{\omega^2}
\end{equation*}
and thus
\begin{equation*}
\chi_1=\frac{\lambda_1\sigma_0^{(a,b,\nu)}(\omega)}{\omega\sigma_1^{(a,b,\nu)}(\omega)}=\frac{c_1^s}{\omega^2}\leq\widetilde{\chi_1}
\end{equation*}
because $c_1^s\leq\widetilde{c_1^s}$. Hence, we obtain that $\chi_n\leq\widetilde{\chi_n}$ for all $n\in\mathbb{N}$ (cf. the theory of chain sequences as recalled in Section~\ref{sec:preliminaries}). We now \textit{claim} that $(\widetilde{\chi_n})_{n\in\mathbb{N}}$ is strictly increasing. Once the claim is proven, we can conclude as follows: since the sequences $(\widetilde{\lambda_n}/\omega^2)_{n\geq N}=(\lambda_n/\omega^2)_{n\geq N}$ and $(\widetilde{\chi_n})_{n\geq N-1}$ are strictly increasing, $(\widetilde{\chi_n})_{n\geq N-1}$ is not the maximal parameter sequence for $(\widetilde{\lambda_n}/\omega^2)_{n\geq N}=(\lambda_n/\omega^2)_{n\geq N}$. Consequently, $(\chi_n)_{n\geq N-1}$ is not the maximal parameter sequence for $(\widetilde{\lambda_n}/\omega^2)_{n\geq N}=(\lambda_n/\omega^2)_{n\geq N}$. Therefore, the infinite products $\prod_{n=N}^{\infty}m_n^{\prime}/\widetilde{\chi_n}$ and $\prod_{n=N}^{\infty}m_n^{\prime}/\chi_n$ converge absolutely, where $(m_n^{\prime})_{n\geq N-1}$ shall denote the minimal parameter sequence for $(\widetilde{\lambda_n}/\omega^2)_{n\geq N}=(\lambda_n/\omega^2)_{n\geq N}$; this establishes the assertion.\\

It is left to establish the claim: Since $(\widetilde{c_n^s})_{n\in\mathbb{N}}\subseteq(0,1/2)$ is strictly increasing, we can apply Theorem~\ref{thm:szwarcturan} and obtain that $(\widetilde{S_n^{(a,b,\nu)}}(x))_{n\in\mathbb{N}_0}$ satisfies Tur\'{a}n's inequality, i.e.,
\begin{align*}
(\widetilde{S_n^{(a,b,\nu)}}(x))^2-\widetilde{S_{n+1}^{(a,b,\nu)}}(x)\widetilde{S_{n-1}^{(a,b,\nu)}}(x)&\geq0\;(x\in[-1,1]),\\
(\widetilde{S_n^{(a,b,\nu)}}(x))^2-\widetilde{S_{n+1}^{(a,b,\nu)}}(x)\widetilde{S_{n-1}^{(a,b,\nu)}}(x)&>0\;(x\in(-1,1))
\end{align*}
for all $n\in\mathbb{N}$. Therefore, we have the estimation
\begin{align*}
\widetilde{\chi_{n+1}}-\widetilde{\chi_n}&=\frac{\widetilde{\lambda_{n+1}}\widetilde{\sigma_n^{(a,b,\nu)}}(\omega)}{\omega\widetilde{\sigma_{n+1}^{(a,b,\nu)}}(\omega)}-\frac{\widetilde{\lambda_n}\widetilde{\sigma_{n-1}^{(a,b,\nu)}}(\omega)}{\omega\widetilde{\sigma_n^{(a,b,\nu)}}(\omega)}=\\
&=\frac{\widetilde{\lambda_{n+1}}(\widetilde{\sigma_n^{(a,b,\nu)}}(\omega))^2-\widetilde{\lambda_n}\widetilde{\sigma_{n+1}^{(a,b,\nu)}}(\omega)\widetilde{\sigma_{n-1}^{(a,b,\nu)}}(\omega)}{\omega\widetilde{\sigma_{n+1}^{(a,b,\nu)}}(\omega)\widetilde{\sigma_n^{(a,b,\nu)}}(\omega)}=\\
&=\frac{\widetilde{c_{n+1}^s}(\widetilde{S_n^{(a,b,\nu)}}(\omega))^2-\widetilde{c_n^s}\widetilde{S_{n+1}^{(a,b,\nu)}}(\omega)\widetilde{S_{n-1}^{(a,b,\nu)}}(\omega)}{\omega\widetilde{S_{n+1}^{(a,b,\nu)}}(\omega)\widetilde{S_n^{(a,b,\nu)}}(\omega)}>\\
&>\frac{\widetilde{c_n^s}(\widetilde{S_n^{(a,b,\nu)}}(\omega))^2-\widetilde{c_n^s}\widetilde{S_{n+1}^{(a,b,\nu)}}(\omega)\widetilde{S_{n-1}^{(a,b,\nu)}}(\omega)}{\omega\widetilde{S_{n+1}^{(a,b,\nu)}}(\omega)\widetilde{S_n^{(a,b,\nu)}}(\omega)}>\\
&>0
\end{align*}
for all $n\in\mathbb{N}$.\\

\textit{Step 3:} Combining Step 1 and Step 2, we see that
\begin{equation*}
\tau:=\frac{\tau_{\omega}}{\tau_1}
\end{equation*}
is as desired.
\end{proof}

Our proof of Theorem~\ref{thm:pwamassocpoll} needs another preparation, which is an extension of \cite[Lemma 4.1]{Ka15}:

\begin{lemma}\label{lma:assocpollasymptotics}
Let $\alpha>-1/2$, $0\leq\lambda<\alpha+1/2$ and $\nu\geq0$. Let
\begin{align*}
\rho&:=\sqrt{1-\left(\frac{2\lambda}{2\alpha+1}\right)^2},\\
\gamma&:=\sqrt{\frac{2\alpha-2\lambda+1}{2\alpha+2\lambda+1}},
\end{align*}
and let
\begin{align*}
s_n&:=\frac{(2\alpha+2\lambda+1)(n+\nu+2\alpha+1)}{(2\alpha+1)(2n+2\nu+2\alpha+2\lambda+1)}\;(n\in\mathbb{N}),\\
t_n&:=1-s_n=\frac{(2\alpha-2\lambda+1)(n+\nu)}{(2\alpha+1)(2n+2\nu+2\alpha+2\lambda+1)}\;(n\in\mathbb{N}).
\end{align*}
Then the recurrence relation $\psi_1:=\rho$,
\begin{equation*}
\psi_{n+1}:=\frac{\rho\psi_n-t_n}{s_n\psi_n}\;(n\in\mathbb{N}),
\end{equation*}
defines a sequence $(\psi_n)_{n\in\mathbb{N}}\subseteq[\gamma,\infty)$ which satisfies
\begin{equation}\label{eq:assocpollasymptotics}
\psi_n\geq\frac{2\lambda n+2\lambda\nu+2\alpha+1}{2\lambda n+2\lambda\nu+2\alpha-2\lambda+1}\gamma
\end{equation}
for each $n\in\mathbb{N}$.
\end{lemma}

\begin{proof}
We modify the proof of \cite[Lemma 4.1]{Ka15} (and give the details for the sake of completeness): since
\begin{equation*}
\frac{2\lambda+2\lambda\nu+2\alpha+1}{2\lambda\nu+2\alpha+1}\gamma=\frac{(2\alpha+2\lambda+1)(2\lambda\nu+2\alpha+1)-4\lambda^2\nu}{(2\alpha+1)(2\lambda\nu+2\alpha+1)}\gamma\leq\frac{2\alpha+2\lambda+1}{2\alpha+1}\gamma=\rho=\psi_1,
\end{equation*}
\eqref{eq:assocpollasymptotics} holds true for $n=1$. Let $n\in\mathbb{N}$ be arbitrary but fixed, assume that $\psi_n$ is well-defined and assume that \eqref{eq:assocpollasymptotics} is satisfied for $n$. Since $\psi_n>0$, $\psi_{n+1}$ is well-defined, too, and it is left to establish that \eqref{eq:assocpollasymptotics} is fulfilled for $n+1$. The latter is equivalent to
\begin{equation*}
\frac{2\lambda n+2\lambda\nu+2\alpha+2\lambda+1}{2\lambda n+2\lambda\nu+2\alpha+1}\gamma\leq\frac{\rho\psi_n-t_n}{s_n\psi_n}
\end{equation*}
or
\begin{equation*}
t_n\leq\left(\rho-\gamma s_n\frac{2\lambda n+2\lambda\nu+2\alpha+2\lambda+1}{2\lambda n+2\lambda\nu+2\alpha+1}\right)\psi_n.
\end{equation*}
Since
\begin{align*}
s_n\frac{2\lambda n+2\lambda\nu+2\alpha+2\lambda+1}{2\lambda n+2\lambda\nu+2\alpha+1}&=s_n\frac{(2\alpha+2\lambda+1)(2\lambda n+2\lambda\nu+2\alpha+1)-2\lambda(2\lambda n+2\lambda\nu)}{(2\alpha+1)(2\lambda n+2\lambda\nu+2\alpha+1)}\leq\\
&\leq s_n\frac{2\alpha+2\lambda+1}{2\alpha+1}<\\
&<\frac{2\alpha+2\lambda+1}{2\alpha+1}=\\
&=\frac{\rho}{\gamma},
\end{align*}
we obtain equivalence to
\begin{equation*}
\psi_n\geq\frac{t_n}{\rho-\gamma s_n\frac{2\lambda n+2\lambda\nu+2\alpha+2\lambda+1}{2\lambda n+2\lambda\nu+2\alpha+1}}.
\end{equation*}
Therefore, it is sufficient to establish that
\begin{equation*}
\frac{2\lambda n+2\lambda\nu+2\alpha+1}{2\lambda n+2\lambda\nu+2\alpha-2\lambda+1}\gamma\geq\frac{t_n}{\rho-\gamma s_n\frac{2\lambda n+2\lambda\nu+2\alpha+2\lambda+1}{2\lambda n+2\lambda\nu+2\alpha+1}}
\end{equation*}
or, equivalently,
\begin{equation*}
\frac{2\lambda n+2\lambda\nu+2\alpha-2\lambda+1}{2\lambda n+2\lambda\nu+2\alpha+1}\frac{t_n}{\gamma\rho}\leq1-\frac{\gamma}{\rho}s_n\frac{2\lambda n+2\lambda\nu+2\alpha+2\lambda+1}{2\lambda n+2\lambda\nu+2\alpha+1}.
\end{equation*}
Since the left-hand side of the latter inequality reduces to
\begin{equation*}
\frac{(n+\nu)(2\lambda n+2\lambda\nu+2\alpha-2\lambda+1)}{(2n+2\nu+2\alpha+2\lambda+1)(2\lambda n+2\lambda\nu+2\alpha+1)}
\end{equation*}
and the right hand side reduces to
\begin{align*}
&1-\frac{(n+\nu+2\alpha+1)(2\lambda n+2\lambda\nu+2\alpha+2\lambda+1)}{(2n+2\nu+2\alpha+2\lambda+1)(2\lambda n+2\lambda\nu+2\alpha+1)}=\\
&=\frac{(n+\nu)(2\lambda n+2\lambda\nu+2\alpha-2\lambda+1+4\lambda^2)}{(2n+2\nu+2\alpha+2\lambda+1)(2\lambda n+2\lambda\nu+2\alpha+1)},
\end{align*}
the induction is finished.
\end{proof}

\begin{proof}[Proof (Theorem~\ref{thm:pwamassocpoll})]
As (iv) is trivial from (iii), we only have to prove the first three assertions. For the case $\lambda=0$, which corresponds to the associated ultraspherical polynomials, the situation concerning weak amenability was completely clarified in \cite[Theorem 5.1]{Ka15} by studying the measure \eqref{eq:muassocpoll} and applying Theorem~\ref{thm:wamnec} and Theorem~\ref{thm:wamsuff}; the analogous situation concerning point and right character amenability was completely clarified in \cite[Theorem 5.1]{Ka15} and \cite[Section 3 p. 31]{Ka16b}, respectively. Observe that if $\lambda=0$, then the measure \eqref{eq:muassocpoll} simplifies because the parameters of the hypergeometric function do no longer depend on $x$ in this special case. Our strategy for $\lambda>0$ will be different and avoid a consideration of $\mu$; in fact, the strategy will be a modification of the special case $\nu=0$ (symmetric Pollaczek polynomials) considered in \cite[Theorem 4.1]{Ka15} and rely on Lemma~\ref{lma:assocpollrelation}, Lemma~\ref{lma:assymptoticrelation} and Lemma~\ref{lma:assocpollasymptotics}.\footnote{Earlier contributions to the special cases $\nu=0$ and $\lambda=0$ were made in \cite{Az10,FLS04,La07,La09a,LP10}.}\\

Let $\lambda>0$ from now on; it remains to show that $\ell^1(h)$ fails to be point amenable (which also rules out weak amenability and right character amenability). We distinguish two cases:\\

\textit{Case 1:} let
\begin{equation*}
\lambda\geq-\alpha-\nu-1+\frac{1}{2}\sqrt{4(\nu+1)(2\alpha+\nu+1)+1}.
\end{equation*}
Then $(2\alpha+2\lambda)^2+8\lambda+8\lambda\nu-1\geq0$ and therefore, due to \eqref{eq:assocpollrec},
\begin{align*}
\frac{1}{4}-c_n a_{n-1}&=\frac{1}{4}-\frac{(n+\nu)(n+\nu+2\alpha)}{(2n+2\nu+2\alpha+2\lambda+1)(2n+2\nu+2\alpha+2\lambda-1)}=\\
&=\frac{(2\alpha+2\lambda)^2+8\lambda n+8\lambda\nu-1}{4(2n+2\nu+2\alpha+2\lambda+1)(2n+2\nu+2\alpha+2\lambda-1)}\geq\\
&\geq\frac{(2\alpha+2\lambda)^2+8\lambda+8\lambda\nu-1}{4(2n+2\nu+2\alpha+2\lambda+1)(2n+2\nu+2\alpha+2\lambda-1)}\geq\\
&\geq0
\end{align*}
for each $n\in\mathbb{N}$. Hence, $\ell^1(h)$ is not point amenable as a consequence of Proposition~\ref{prp:kahlerpoint} (i).\\

\textit{Case 2:} let
\begin{equation*}
\lambda<-\alpha-\nu-1+\frac{1}{2}\sqrt{4(\nu+1)(2\alpha+\nu+1)+1}.
\end{equation*}
We first observe that $\lambda<\alpha+1/2$ in this case, which is an immediate consequence of the inequality
\begin{equation*}
4(\nu+1)(2\alpha+\nu+1)+1=(4\alpha+2\nu+3)^2-4(2\alpha+1)(2\alpha+\nu+1)<(4\alpha+2\nu+3)^2.
\end{equation*}
Now let
\begin{equation*}
a:=\frac{2\alpha+2\lambda+1}{2\alpha-2\lambda+1},\;b:=(2\alpha+1)\frac{2\alpha+2\lambda+1}{2\alpha-2\lambda+1}
\end{equation*}
and
\begin{align*}
S_n(x)&:=S_n^{(a,b,\nu)}(x)\;(n\in\mathbb{N}_0),\\
\widetilde{S_n}(x)&:=\widetilde{S_n^{(a,b,\nu)}}(x)\;(n\in\mathbb{N}_0).
\end{align*}
In the following, we use the notation of Lemma~\ref{lma:assocpollasymptotics}. Observe that, directly from the definition (Section~\ref{sec:rw}), $\widetilde{c_n^s}=t_n$ and $\widetilde{a_n^s}=s_n$ for all $n\in\mathbb{N}$. Moreover, we have
\begin{equation*}
P_n^{\prime}(0)=\rho\frac{S_n^{\prime}(0)}{S_n(\rho)}\;(n\in\mathbb{N}_0)
\end{equation*}
due to Lemma~\ref{lma:assocpollrelation}, and we have $\omega=\rho$. We now consider the asymptotic behavior and compare the decay of $S_n^{\prime}(0)$ with the growth of $1/S_n(\rho)$. We \textit{claim} that
\begin{enumerate}[A)]
\item $S_n^{\prime}(0)=\mathcal{O}(n\gamma^n)\;(n\to\infty)$,
\item $1/S_n(\rho)=\mathcal{O}(n^{-1}\gamma^{-n})\;(n\to\infty)$.
\end{enumerate}
Once the claim is established, we obtain that $\{P_n^{\prime}(0):n\in\mathbb{N}_0\}$ is bounded, so $\ell^1(h)$ fails to be point amenable due to Theorem~\ref{thm:lasserpoint}. The validity of the claimed assertions can be seen as follows:
\begin{enumerate}[A)]
\item The recurrence relation for $(S_n(x))_{n\in\mathbb{N}_0}$ yields $a_{2n-1}^s|S_{2n}(0)|=c_{2n-1}^s|S_{2n-2}(0)|\;(n\in\mathbb{N})$, so
\begin{equation}\label{eq:s2n0est}
|S_{2n}(0)|=\prod_{k=1}^n\frac{c_{2k-1}^s}{a_{2k-1}^s}\leq\prod_{k=1}^n\frac{\widetilde{c_{2k-1}^s}}{\widetilde{a_{2k-1}^s}}=\prod_{k=1}^n\frac{2k+\nu-1}{a(2k+\nu-1)+b}<\frac{1}{a^n}
\end{equation}
for all $n\in\mathbb{N}$, where we have used that $c_{2k-1}^s\leq\widetilde{c_{2k-1}^s}$ and $a_{2k-1}^s\geq\widetilde{a_{2k-1}^s}$ for all $k\in\{1,\ldots,n\}$ (see Section~\ref{sec:rw}). Moreover, for each $n\in\mathbb{N}$ we have
\begin{equation*}
x S_{2n}^{\prime}(x)+S_{2n}(x)=a_{2n}^s S_{2n+1}^{\prime}(x)+c_{2n}^s S_{2n-1}^{\prime}(x)
\end{equation*}
and consequently
\begin{align*}
a^n|S_{2n+1}^{\prime}(0)|&\leq\frac{a^n}{a_{2n}^s}|S_{2n}(0)|+\frac{a^n c_{2n}^s}{a_{2n}^s}|S_{2n-1}^{\prime}(0)|\leq\\
&\leq\frac{a^n}{\widetilde{a_{2n}^s}}|S_{2n}(0)|+\frac{a^n\widetilde{c_{2n}^s}}{\widetilde{a_{2n}^s}}|S_{2n-1}^{\prime}(0)|\leq\\
&\leq2a^n|S_{2n}(0)|+\frac{a^n(2n+\nu)}{a(2n+\nu)+b}|S_{2n-1}^{\prime}(0)|<\\
&<2+a^{n-1}|S_{2n-1}^{\prime}(0)|
\end{align*}
due to \eqref{eq:s2n0est}, so induction yields
\begin{equation*}
a^n|S_{2n+1}^{\prime}(0)|\leq2n+1\;(n\in\mathbb{N}_0).
\end{equation*}
Since $S_{2n}^{\prime}(0)\equiv0$ due to symmetry, and since $1/a=\gamma^2$, part A) is established.
\item Since $\omega=\rho$, Lemma~\ref{lma:assymptoticrelation} implies that $(S_n(\rho)/\widetilde{S_n(\rho)})_{n\in\mathbb{N}_0}$ converges to a positive real number. Hence, it suffices to prove that
\begin{equation}\label{eq:Bsimplified}
\frac{1}{\widetilde{S_n}(\rho)}=\mathcal{O}(n^{-1}\gamma^{-n})\;(n\to\infty).
\end{equation}
To do so, we use induction to show that
\begin{equation}\label{eq:Bsimplifiedind}
\widetilde{S_n}(\rho)=\prod_{k=1}^n\psi_k
\end{equation}
for all $n\in\mathbb{N}$ (where $(\psi_n)_{n\in\mathbb{N}}$ is as in Lemma~\ref{lma:assocpollasymptotics}). \eqref{eq:Bsimplifiedind} is obviously true for $n=1$, so let $n\in\mathbb{N}$ be arbitrary but fixed and assume that \eqref{eq:Bsimplifiedind} is satisfied for $1,\ldots,n$. Then
\begin{equation*}
\widetilde{S_{n+1}}(\rho)=\frac{\rho\frac{\widetilde{S_n}(\rho)}{\widetilde{S_{n-1}}(\rho)}-t_n}{s_n\frac{\widetilde{S_n}(\rho)}{\widetilde{S_{n-1}}(\rho)}}\prod_{k=1}^n\psi_k=\frac{\rho\psi_n-t_n}{s_n\psi_n}\prod_{k=1}^n\psi_k=\prod_{k=1}^{n+1}\psi_k.
\end{equation*}
Now combining \eqref{eq:Bsimplifiedind} and Lemma~\ref{lma:assocpollasymptotics}, we obtain
\begin{equation*}
\widetilde{S_n}(\rho)\geq\gamma^n\prod_{k=1}^n\frac{2\lambda k+2\lambda\nu+2\alpha+1}{2\lambda k+2\lambda\nu+2\alpha-2\lambda+1}=\frac{2\lambda n+2\lambda\nu+2\alpha+1}{2\lambda\nu+2\alpha+1}\gamma^n\;(n\in\mathbb{N}),
\end{equation*}
which establishes \eqref{eq:Bsimplified} and finishes the proof.
\end{enumerate}
\end{proof}

\begin{remark}
Concerning Case 2, we note that if $\nu=0$, then B) can be established without the use of Lemma~\ref{lma:assocpollasymptotics} via the following variant: \cite[Section 6, in part. (6.30)]{AI84} yields
\begin{equation*}
\frac{1}{S_n(\rho)}=\Theta(n^{\alpha+\frac{1}{4}}\gamma^{-n}e^{-\sqrt{8\lambda n}})\;(n\to\infty),
\end{equation*}
which obviously implies B) (cf. also \cite[Section 3.3]{Ka16b}). The cited ingredient \cite[(6.30)]{AI84} relies on Perron's formula in the complex plane \cite[Theorem 8.22.3]{Sz75} and is therefore considerably less elementary than our proof of Lemma~\ref{lma:assocpollasymptotics} given above (which, moreover, does not restrict to $\nu=0$). We are not aware of an explicit formula for $P_n^{\prime}(0)$.
\end{remark}

\begin{remark}
Note that in the proof of Theorem~\ref{thm:pwamassocpoll} the sequences $(S_n^{(a,b,\nu)}(x))_{n\in\mathbb{N}_0}$ and $(\widetilde{S_n^{(a,b,\nu)}}(x))_{n\in\mathbb{N}_0}$ served as auxiliary tools (to clarify the amenability properties which correspond to the associated symmetric Pollaczek polynomials) but that we have not yet tackled the amenability properties which correspond to these sequences themselves. The latter is much easier, however: assume that $(P_n(x))_{n\in\mathbb{N}_0}$ is $(S_n^{(a,b,\nu)}(x))_{n\in\mathbb{N}_0}$ or $(\widetilde{S_n^{(a,b,\nu)}}(x))_{n\in\mathbb{N}_0}$, with arbitrary $a>1$, $b>0$, $\nu\geq0$. Then Proposition~\ref{prp:kahlerpoint} (ii) immediately implies that $0$ admits a nonzero bounded point derivation; hence, $\ell^1(h)$ is not even point amenable. For the special case $\nu=0$, this was already observed in \cite[Proposition 4.1]{Ka15}.
\end{remark}

\section{Extended parameter region}\label{sec:extended}

The associated symmetric Pollaczek polynomials can be defined for larger parameter regions than considered in \cite{La94}, cf. \cite[Chapter VI 5]{Ch78}. We show the following:

\begin{theorem}\label{thm:extended}
Let $\lambda\geq0$, $\nu>0$ and
\begin{equation*}
\alpha\in\left(-\frac{1}{2}-\frac{\nu}{2},-\frac{1}{2}\right].
\end{equation*}
Then $c_0=0$ and \eqref{eq:assocpollrec} define a sequence $(c_n)_{n\in\mathbb{N}}\subseteq(0,1)$. If one defines $(P_n(x))_{n\in\mathbb{N}_0}=:(Q_n^{(\alpha,\lambda,\nu)}(x))_{n\in\mathbb{N}_0}$ via $a_n\equiv1-c_n$, the assertions of Theorem~\ref{thm:assocpollnonneg} and Corollary~\ref{cor:assocpollnonneg} remain valid. Moreover, $\ell^1(h)$ is not point amenable (and hence not weakly amenable, not right character amenable and not amenable).
\end{theorem}

\begin{proof}
Let $\phi:[1,\infty)\rightarrow(0,\infty)$ be defined as in \eqref{eq:phidef}, and let $\eta:\mathbb{R}\rightarrow\mathbb{R}$ be defined as in \eqref{eq:etadef}. Since
\begin{equation*}
\frac{1}{4}-\phi(x)=\frac{(2\alpha+2\lambda)^2+8\lambda x+8\lambda\nu-1}{4(2x+2\nu+2\alpha+2\lambda+1)(2x+2\nu+2\alpha+2\lambda-1)}
\end{equation*}
and
\begin{equation}\label{eq:assocpollextensionest}
(2\alpha+2\lambda)^2+8\lambda+8\lambda\nu-1\geq(2\alpha+2\lambda)^2+8\lambda+8\lambda(-2\alpha-1)-1=(2\alpha-2\lambda)^2-1\geq0,
\end{equation}
$\phi$ maps into $(0,1/4]$. Hence, Wall's comparison test \cite[Theorem III-5.7]{Ch78} yields that $(\phi(n))_{n\in\mathbb{N}}$ is a chain sequence with minimal parameter sequence $(c_n)_{n\in\mathbb{N}_0}$. In particular, $(c_n)_{n\in\mathbb{N}}\subseteq(0,1)$ is well-defined. We now show that $\phi$ is nondecreasing, which allows us to conclude as in Section~\ref{sec:Turan} that $(c_n)_{n\in\mathbb{N}}$ is strictly increasing and bounded from above by $1/2$ (so property (P) and Tur\'{a}n's inequality follow from Szwarc's criteria Theorem~\ref{thm:szwarcnonneg} and Theorem~\ref{thm:szwarcturan}). If $\lambda=0$, then
\begin{equation*}
\phi^{\prime}(x)=\frac{(2\alpha+1)(2\alpha-1)(2x+2\nu+2\alpha)}{(2x+2\nu+2\alpha+1)^2(2x+2\nu+2\alpha-1)^2}\geq0
\end{equation*}
for all $x\in[1,\infty)$. If $\lambda>0$, however, then \eqref{eq:assocpollextensionest} yields
\begin{align*}
\phi^{\prime}(x)&=\frac{\eta(x)}{8\lambda(2x+2\nu+2\alpha+2\lambda+1)^2(2x+2\nu+2\alpha+2\lambda-1)^2}=\\
&=\frac{(8\lambda(x+\nu)+(2\alpha+2\lambda)^2-1)^2-(1-(2\alpha-2\lambda)^2)(1-(2\alpha+2\lambda)^2)}{8\lambda(2x+2\nu+2\alpha+2\lambda+1)^2(2x+2\nu+2\alpha+2\lambda-1)^2}\geq\\
&\geq\frac{((2\alpha-2\lambda)^2-1)^2-(1-(2\alpha-2\lambda)^2)(1-(2\alpha+2\lambda)^2)}{8\lambda(2x+2\nu+2\alpha+2\lambda+1)^2(2x+2\nu+2\alpha+2\lambda-1)^2}=\\
&=-\frac{2\alpha((2\alpha-2\lambda)^2-1)}{(2x+2\nu+2\alpha+2\lambda+1)^2(2x+2\nu+2\alpha+2\lambda-1)^2}\geq\\
&\geq0
\end{align*}
for all $x\in[1,\infty)$. The proof of Corollary~\ref{cor:assocpollnonneg} can be copied. Finally, the failure of point amenability follows from $c_n a_{n-1}=\phi(n)\leq\frac{1}{4}\;(n\in\mathbb{N})$ and Proposition~\ref{prp:kahlerpoint} (i), or from \eqref{eq:assocpollextensionest} just as in Case 1 of the proof of Theorem~\ref{thm:pwamassocpoll}.
\end{proof}

\appendix

\section{Alternative proof of Theorem~\ref{thm:assocpollnonneg} via Tur\'{a}n's inequality for Laguerre polynomials}

As announced in Section~\ref{sec:Turan}, we consider Theorem~\ref{thm:assocpollnonneg} again and present an alternative proof for the subcase $0<\lambda<-|\alpha|+1/2$. This alternative proof avoids the transformation provided by Lemma~\ref{lma:assocpollrelation}; instead, it is based on Tur\'{a}n's inequality for the (renormalized) Laguerre polynomials $\left(\frac{L_n^{(2\alpha,0)}(x)}{L_n^{(2\alpha,0)}(0)}\right)_{n\in\mathbb{N}_0}$ (cf. the proof of Lemma~\ref{lma:cest} below) and on the behavior of $\phi:[1,\infty)\rightarrow(0,\infty)$ given by \eqref{eq:phidef}. We first introduce some additional notation and further auxiliary functions. Let $\alpha>-1/2$, $\lambda\geq0$ and $\nu\geq0$. In the following, we use an additional superscript ``$(\nu=0)$'' when referring to $\nu=0$, with the remaining parameters $\alpha$ and $\lambda$ unchanged. More precisely, while $(c_n)_{n\in\mathbb{N}}$ refers to $(Q_n^{(\alpha,\lambda,\nu)}(x))_{n\in\mathbb{N}_0}$, $(c_n^{(\nu=0)})_{n\in\mathbb{N}}$ refers to $(Q_n^{(\alpha,\lambda,0)}(x))_{n\in\mathbb{N}_0}$, and so on. In the same way, we use the superscripts ``$(\lambda=0)$'' and ``$(\lambda=\nu=0)$''. One has
\begin{equation}\label{eq:cultrasph}
c_n^{(\lambda=\nu=0)}=\frac{n}{2n+2\alpha+1}\;(n\in\mathbb{N})
\end{equation}
\cite{La94}. Observe that
\begin{align*}
\phi^{(\lambda=0)}(x)-\phi(x)&=\frac{1}{(2x+2\nu+2\alpha+1)(2x+2\nu+2\alpha-1)}\\
&\quad\times\frac{4\lambda(x+\nu)(x+\nu+2\alpha)(2x+2\nu+2\alpha+\lambda)}{(2x+2\nu+2\alpha+2\lambda+1)(2x+2\nu+2\alpha+2\lambda-1)}
\end{align*}
for all $x\in[1,\infty)$, which implies the identity
\begin{equation}\label{eq:phiassocultrasphest}
\phi(x)\leq\phi^{(\lambda=0)}(x)\;(x\in[1,\infty)).
\end{equation}

From now on, let $0<\lambda<-|\alpha|+1/2$.\\

Since $|\alpha|<1/2-\lambda<1/2+\lambda$, we have
\begin{equation*}
(1-(2\alpha-2\lambda)^2)(1-(2\alpha+2\lambda)^2)=((1+2\lambda)^2-4\alpha^2)((1-2\lambda)^2-4\alpha^2)>0.
\end{equation*}
For each $x\in[1,\infty)$, we compute
\begin{align*}
\phi^{(\nu=0)}(x)-\phi(x)&=-\frac{\nu}{8\lambda}\frac{1}{(2x+2\alpha+2\lambda+1)(2x+2\alpha+2\lambda-1)}\\
&\quad\times\frac{1}{(2x+2\nu+2\alpha+2\lambda+1)(2x+2\nu+2\alpha+2\lambda-1)}\theta(x)
\end{align*}
with $\theta:\mathbb{R}\rightarrow\mathbb{R}$,
\begin{equation*}
\theta(x):=\left(8\lambda\left(x+\frac{\nu}{2}\right)+(2\alpha+2\lambda)^2-1\right)^2-(1-(2\alpha-2\lambda)^2)(1-(2\alpha+2\lambda)^2)-16\lambda^2\nu^2.
\end{equation*}
The zeros of $\eta$ \eqref{eq:etadef} and $\theta$ are given by
\begin{equation*}
-\nu+\frac{1-(2\alpha+2\lambda)^2\pm\sqrt{(1-(2\alpha-2\lambda)^2)(1-(2\alpha+2\lambda)^2)}}{8\lambda}
\end{equation*}
and
\begin{equation*}
-\frac{\nu}{2}+\frac{1-(2\alpha+2\lambda)^2\pm\sqrt{(1-(2\alpha-2\lambda)^2)(1-(2\alpha+2\lambda)^2)+16\lambda^2\nu^2}}{8\lambda},
\end{equation*}
respectively. Observe that
\begin{equation}\label{eq:estsmallerzeros}
\frac{1-(2\alpha+2\lambda)^2-\sqrt{(1-(2\alpha-2\lambda)^2)(1-(2\alpha+2\lambda)^2)}}{8\lambda}<1.
\end{equation}
This follows because
\begin{align*}
&(1-(2\alpha-2\lambda)^2)(1-(2\alpha+2\lambda)^2)-(1-(2\alpha+2\lambda)^2-8\lambda)^2=\\
&=16(\alpha+1)\lambda\left(1-(2\alpha+2\lambda)^2-\frac{4\lambda}{\alpha+1}\right)>\\
&>16(\alpha+1)\lambda(1-(2\alpha+2\lambda)^2-8\lambda);
\end{align*}
hence, if $1-(2\alpha+2\lambda)^2-8\lambda\geq0$, then $1-(2\alpha+2\lambda)^2-8\lambda<\sqrt{(1-(2\alpha-2\lambda)^2)(1-(2\alpha+2\lambda)^2)}$. Now let
\begin{equation*}
x_{\ast}:=-\nu+\frac{1-(2\alpha+2\lambda)^2+\sqrt{(1-(2\alpha-2\lambda)^2)(1-(2\alpha+2\lambda)^2)}}{8\lambda}
\end{equation*}
and
\begin{equation*}
x_{\ast\ast}:=-\frac{\nu}{2}+\frac{1-(2\alpha+2\lambda)^2+\sqrt{(1-(2\alpha-2\lambda)^2)(1-(2\alpha+2\lambda)^2)+16\lambda^2\nu^2}}{8\lambda}.
\end{equation*}
As a consequence of \eqref{eq:estsmallerzeros}, we see that $\phi^{\prime}$ has at last one zero, and the potential zero is given by $x_{\ast}$. Moreover, \eqref{eq:estsmallerzeros} yields that either $\phi^{(\nu=0)}-\phi\equiv0$ (which corresponds to the special case $\nu=0$) or $\phi^{(\nu=0)}-\phi$ has at last one zero; in the latter case, the potential zero is given by $x_{\ast\ast}$. Obviously, one has $x_{\ast\ast}\geq x_{\ast}$ (with equality if and only if $\nu=0$, i.e., in the symmetric Pollaczek case) and $x_{\ast\ast}>0$.\\

As a consequence of the preceding observations, we see that
\begin{equation}\label{eq:phipollest}
\phi(x)\leq\phi^{(\nu=0)}(x)\;(x\in[1,x_{\ast\ast}])
\end{equation}
if $x_{\ast\ast}\geq1$.\\

Assume that $x_{\ast}>1$. The monotonicity and limiting behavior of $\phi$ shows that $\phi(x_{\ast})<1/4$, and there is a (uniquely determined) $x_0\in[1,x_{\ast})$ such that $\phi(x)>1/4$ for all $x<x_0$ and $\phi(x)\leq1/4$ for all $x\geq x_0$. $\phi$ is strictly decreasing on $[1,x_{\ast}]$ and strictly increasing on $[x_{\ast},\infty)$. Hence, $\phi(x)<1/4$ for all $x\in(x_0,\infty)$. Defining $\xi:[x_0,\infty)\rightarrow(0,1/2]$ via
\begin{equation}\label{eq:xidef}
\xi(x):=\frac{1}{2}(1-\sqrt{1-4\phi(x)}),
\end{equation}
we have
\begin{equation*}
\xi^{\prime}(x)=\frac{\phi^{\prime}(x)}{\sqrt{1-4\phi(x)}}
\end{equation*}
for all $x\in(x_0,\infty)$; consequently, $\xi$ is strictly decreasing on $[x_0,x_{\ast}]$ and strictly increasing on $[x_{\ast},\infty)$.\\

Finally, let $\lambda\geq0$ be arbitrary again and $\iota:[1,\infty)\rightarrow\mathbb{R}$ be defined by
\begin{equation*}
\iota(x):=\frac{1}{2}\left(1-\frac{\sqrt{8\lambda x+(2\alpha+2\lambda+1)^2}}{2x+2\alpha+2\lambda+1}\right).
\end{equation*}
$\iota^{\prime}$ is given by
\begin{equation*}
\iota^{\prime}(x)=\frac{4\lambda x+4\left(\alpha+\lambda+\frac{1}{2}\right)\left(\alpha+\frac{1}{2}\right)}{(2x+2\alpha+2\lambda+1)^2\sqrt{8\lambda x+(2\alpha+2\lambda+1)^2}},
\end{equation*}
which shows that $\iota$ is strictly increasing.

\begin{lemma}\label{lma:cest}
Let $\alpha>-1/2$, $\lambda,\nu\geq0$ and $P_n(x)=Q_n^{(\alpha,\lambda,\nu)}(x)\;(n\in\mathbb{N}_0)$. Then the following hold:
\begin{enumerate}[(i)]
\item $c_n\leq c_n^{(\lambda=0)}$ for all $n\in\mathbb{N}_0$,
\item if $0<\lambda<-|\alpha|+1/2$, then $c_n\leq c_n^{(\nu=0)}$ for all $n\in\{0,\ldots,\lfloor x_{\ast\ast}\rfloor\}$,
\item $c_n^{(\nu=0)}\leq\iota(n)$ for all $n\in\mathbb{N}$. If $\lambda>0$, then the inequality is always strict, and if $\lambda=0$, then one always has equality.
\end{enumerate}
\end{lemma}

\begin{proof}
\begin{enumerate}[(i)]
\item We use induction on $n$: the case $n=0$ is clear. If $n\in\mathbb{N}_0$ is arbitrary but fixed and $c_n\leq c_n^{(\lambda=0)}$, then, as a consequence of \eqref{eq:assocpollrecmod} and \eqref{eq:phiassocultrasphest},
\begin{equation*}
c_{n+1}=\frac{\phi(n+1)}{1-c_n}\leq\frac{\phi^{(\lambda=0)}(n+1)}{1-c_n}\leq\frac{\phi^{(\lambda=0)}(n+1)}{1-c_n^{(\lambda=0)}}=c_{n+1}^{(\lambda=0)}.
\end{equation*}
\item Since the case $x_{\ast\ast}<1$ is trivial, we may assume that $x_{\ast\ast}\geq1$. Then \eqref{eq:phipollest} yields
\begin{equation*}
\phi(n)\leq\phi^{(\nu=0)}(n)\;(n\in\{1,\ldots,\lfloor x_{\ast\ast}\rfloor\}).
\end{equation*}
Now the assertion follows as in (i): the case $n=0$ is clear, and if $n\in\{0,\ldots,\lfloor x_{\ast\ast}\rfloor-1\}$ is arbitrary but fixed and $c_n\leq c_n^{(\nu=0)}$, then, as a consequence of \eqref{eq:assocpollrecmod},
\begin{equation*}
c_{n+1}=\frac{\phi(n+1)}{1-c_n}\leq\frac{\phi^{(\nu=0)}(n+1)}{1-c_n}\leq\frac{\phi^{(\nu=0)}(n+1)}{1-c_n^{(\nu=0)}}=c_{n+1}^{(\nu=0)}.
\end{equation*}
\item (i) and \eqref{eq:cultrasph} yield that
\begin{equation}\label{eq:csmalleronehalf}
c_n^{(\nu=0)}\leq c_n^{(\lambda=\nu=0)}<\frac{1}{2}\;(n\in\mathbb{N}).
\end{equation}
Let $\lambda>0$. Via \eqref{eq:assocpollrecmod} and \eqref{eq:estimationforappendix} (which was obtained in Section~\ref{sec:Turan} as a consequence of Tur\'{a}n's inequality for Laguerre polynomials \eqref{eq:LaguerreTuranfirst}, \eqref{eq:LaguerreTuransecond}), we have
\begin{align*}
\frac{(n+1)(2n+2\alpha+2\lambda+1)}{n(2n+2\alpha+2\lambda+3)}c_n^{(\nu=0)}&<c_{n+1}^{(\nu=0)}=\\
&=\frac{\phi^{(\nu=0)}(n+1)}{1-c_n^{(\nu=0)}}=\\
&=\frac{(n+1)(n+2\alpha+1)}{(2n+2\alpha+2\lambda+3)(2n+2\alpha+2\lambda+1)}\frac{1}{1-c_n^{(\nu=0)}}
\end{align*}
and therefore
\begin{equation*}
(c_n^{(\nu=0)})^2-c_n^{(\nu=0)}+\underbrace{\frac{n(n+2\alpha+1)}{(2n+2\alpha+2\lambda+1)^2}}_{\in\left(0,\frac{1}{4}\right)}>0.
\end{equation*}
Due to \eqref{eq:csmalleronehalf}, we now can conclude that
\begin{equation*}
c_n^{(\nu=0)}<\frac{1}{2}\left(1-\sqrt{1-4\frac{n(n+2\alpha+1)}{(2n+2\alpha+2\lambda+1)^2}}\right)=\iota(n)
\end{equation*}
for all $n\in\mathbb{N}$. The case $\lambda=0$ can be shown in the same way.
\end{enumerate}
\end{proof}

\begin{proof}[Proof (Theorem~\ref{thm:assocpollnonneg}, variant for the case $0<\lambda<-|\alpha|+1/2$)]
Let $0<\lambda<-|\alpha|+1/2$ (these conditions particularly imply that $|\alpha|<1/2$). If $x_{\ast}\leq1$, we have $\eta(x)>0$ for all $x\in(1,\infty)$, so $\phi$ is strictly increasing; consequently, $(c_n)_{n\in\mathbb{N}}$ is strictly increasing (cf. above). Hence, it remains to consider the case $x_{\ast}>1$, which shall be assumed from now on. We have to show that $(c_n)_{n\in\mathbb{N}}$ is strictly increasing again. Let $x_0$ and the function $\xi:[x_0,\infty)\rightarrow(0,1/2]$ be defined as above \eqref{eq:xidef}. The proof will be done in three steps:\\

\textit{Step 1:} we show that
\begin{equation}\label{eq:crucialcest}
c_{\lfloor x_{\ast}\rfloor}<\xi(x_{\ast}).
\end{equation}
To establish \eqref{eq:crucialcest}, we first compute
\begin{align*}
64\lambda^2((x_{\ast}+\nu+\alpha)^2-\alpha^2)&=(8\lambda x_{\ast}+8\lambda\nu+8\alpha\lambda)^2-64\alpha^2\lambda^2=\\
&=(1-4\alpha^2-4\lambda^2+\sqrt{(1-(2\alpha-2\lambda)^2)(1-(2\alpha+2\lambda)^2)})^2-64\alpha^2\lambda^2,
\end{align*}
\begin{align*}
0&<64\lambda^2((2x_{\ast}+2\nu+2\alpha+2\lambda)^2-1)=\\
&=4(8\lambda x_{\ast}+8\lambda\nu+8\alpha\lambda+8\lambda^2)^2-64\lambda^2=\\
&=4(1-4\alpha^2+4\lambda^2+\sqrt{(1-(2\alpha-2\lambda)^2)(1-(2\alpha+2\lambda)^2)})^2-64\lambda^2
\end{align*}
and consequently
\begin{align*}
\phi(x_{\ast})&=\frac{(x_{\ast}+\nu)(x_{\ast}+\nu+2\alpha)}{(2x_{\ast}+2\nu+2\alpha+2\lambda+1)(2x_{\ast}+2\nu+2\alpha+2\lambda-1)}=\\
&=\frac{(x_{\ast}+\nu+\alpha)^2-\alpha^2}{(2x_{\ast}+2\nu+2\alpha+2\lambda)^2-1}=\\
&=\frac{(1-4\alpha^2-4\lambda^2+\sqrt{(1-(2\alpha-2\lambda)^2)(1-(2\alpha+2\lambda)^2)})^2-64\alpha^2\lambda^2}{4(1-4\alpha^2+4\lambda^2+\sqrt{(1-(2\alpha-2\lambda)^2)(1-(2\alpha+2\lambda)^2)})^2-64\lambda^2},
\end{align*}
which yields that
\begin{equation}\label{eq:longcalcfirst}
1-4\phi(x_{\ast})=\frac{16\lambda^2\sqrt{(1-(2\alpha-2\lambda)^2)(1-(2\alpha+2\lambda)^2)}}{(1-4\alpha^2+4\lambda^2+\sqrt{(1-(2\alpha-2\lambda)^2)(1-(2\alpha+2\lambda)^2)})^2-16\lambda^2};
\end{equation}
the denominator in \eqref{eq:longcalcfirst} is positive. Next, we compute
\begin{equation*}
8\lambda x_{\ast}^{(\nu=0)}+(2\alpha+2\lambda+1)^2=4\alpha+4\lambda+2+\sqrt{(1-(2\alpha-2\lambda)^2)(1-(2\alpha+2\lambda)^2)}
\end{equation*}
and (since $x_{\ast}^{(\nu=0)}\geq x_{\ast}>1$)
\begin{equation*}
0<(2x_{\ast}^{(\nu=0)}+2\alpha+2\lambda+1)^2=\frac{(1-4\alpha^2+4\lambda^2+4\lambda+\sqrt{(1-(2\alpha-2\lambda)^2)(1-(2\alpha+2\lambda)^2)})^2}{16\lambda^2},
\end{equation*}
which implies that
\begin{equation}\label{eq:longcalcsecond}
\frac{8\lambda x_{\ast}^{(\nu=0)}+(2\alpha+2\lambda+1)^2}{(2x_{\ast}^{(\nu=0)}+2\alpha+2\lambda+1)^2}=\frac{16\lambda^2(4\alpha+4\lambda+2+\sqrt{(1-(2\alpha-2\lambda)^2)(1-(2\alpha+2\lambda)^2)})}{(1-4\alpha^2+4\lambda^2+4\lambda+\sqrt{(1-(2\alpha-2\lambda)^2)(1-(2\alpha+2\lambda)^2)})^2};
\end{equation}
the denominator in \eqref{eq:longcalcsecond} is positive. Since
\begin{align*}
&(4\alpha+4\lambda+2+\sqrt{(1-(2\alpha-2\lambda)^2)(1-(2\alpha+2\lambda)^2)})\\
&\quad\quad\times((1-4\alpha^2+4\lambda^2+\sqrt{(1-(2\alpha-2\lambda)^2)(1-(2\alpha+2\lambda)^2)})^2-16\lambda^2)\\
&\quad-\sqrt{(1-(2\alpha-2\lambda)^2)(1-(2\alpha+2\lambda)^2)}\\
&\quad\quad\quad\times(1-4\alpha^2+4\lambda^2+4\lambda+\sqrt{(1-(2\alpha-2\lambda)^2)(1-(2\alpha+2\lambda)^2)})^2=\\
&=4((1+2\alpha)^2-4\lambda^2)\\
&\quad\times((1-2\alpha)\sqrt{(1-(2\alpha-2\lambda)^2)(1-(2\alpha+2\lambda)^2)}+(2\alpha+1)((1-2\alpha)^2-4\lambda^2))=\\
&=4((1+2\alpha)^2-4\lambda^2)\\
&\quad\times((1-2\alpha)\sqrt{((1+2\alpha)^2-4\lambda^2)((1-2\alpha)^2-4\lambda^2)}+(2\alpha+1)((1-2\alpha)^2-4\lambda^2))>\\
&>0,
\end{align*}
where we have used that $2\lambda<1+2\alpha$, $\alpha<1/2$ and $2\lambda<1-2\alpha$, \eqref{eq:longcalcfirst} and \eqref{eq:longcalcsecond} yield
\begin{equation*}
0<1-4\phi(x_{\ast})<\frac{8\lambda x_{\ast}^{(\nu=0)}+(2\alpha+2\lambda+1)^2}{(2x_{\ast}^{(\nu=0)}+2\alpha+2\lambda+1)^2}.
\end{equation*}
Combining the latter inequality with Lemma~\ref{lma:cest} (and taking into account that $x_{\ast}\leq x_{\ast\ast}$), we can conclude that
\begin{equation*}
\xi(x_{\ast})>\iota(x_{\ast}^{(\nu=0)})\geq\iota(x_{\ast})\geq\iota(\lfloor x_{\ast}\rfloor)>c_{\lfloor x_{\ast}\rfloor}^{(\nu=0)}\geq c_{\lfloor x_{\ast}\rfloor},
\end{equation*}
which establishes \eqref{eq:crucialcest}.\\

\textit{Step 2:} we now use induction on $n$ to show that
\begin{equation}\label{eq:cestindfirst}
c_{\lfloor x_{\ast}\rfloor+n}<c_{\lfloor x_{\ast}\rfloor+n+1}
\end{equation}
for all $n\in\mathbb{N}_0$. As a consequence of Step 1, we have $c_{\lfloor x_{\ast}\rfloor}<\xi(\lfloor x_{\ast}\rfloor+1)$ and therefore, due to \eqref{eq:assocpollrecmod},
\begin{equation*}
(c_{\lfloor x_{\ast}\rfloor+1}-c_{\lfloor x_{\ast}\rfloor})(1-c_{\lfloor x_{\ast}\rfloor})=c_{\lfloor x_{\ast}\rfloor}^2-c_{\lfloor x_{\ast}\rfloor}+\phi(\lfloor x_{\ast}\rfloor+1)>0.
\end{equation*}
This establishes \eqref{eq:cestindfirst} for $n=0$. Now let $n\in\mathbb{N}_0$ be arbitrary but fixed and assume that \eqref{eq:cestindfirst} is true for $n$. Then
\begin{equation*}
c_{\lfloor x_{\ast}\rfloor+n+2}=\frac{\phi(\lfloor x_{\ast}\rfloor+n+2)}{1-c_{\lfloor x_{\ast}\rfloor+n+1}}>\frac{\phi(\lfloor x_{\ast}\rfloor+n+1)}{1-c_{\lfloor x_{\ast}\rfloor+n+1}}>\frac{\phi(\lfloor x_{\ast}\rfloor+n+1)}{1-c_{\lfloor x_{\ast}\rfloor+n}}=c_{\lfloor x_{\ast}\rfloor+n+1}
\end{equation*}
by \eqref{eq:assocpollrecmod}, so \eqref{eq:cestindfirst} is also valid for $n+1$.\\

\textit{Step 3:} finally, we use induction on $n$ to show that
\begin{equation}\label{eq:cestindsecond}
c_{\lfloor x_{\ast}\rfloor-n-1}<c_{\lfloor x_{\ast}\rfloor-n}
\end{equation}
for all $n\in\{0,\ldots,\lfloor x_{\ast}\rfloor-1\}$. Concerning the initial step, we distinguish two cases: if $\lfloor x_{\ast}\rfloor\geq x_0$, then $\xi(\lfloor x_{\ast}\rfloor)$ is defined and we have $c_{\lfloor x_{\ast}\rfloor}<\xi(\lfloor x_{\ast}\rfloor)$ as a consequence of Step 1. Hence, we get
\begin{align*}
(c_{\lfloor x_{\ast}\rfloor}-c_{\lfloor x_{\ast}\rfloor-1})(1-c_{\lfloor x_{\ast}\rfloor-1})&=c_{\lfloor x_{\ast}\rfloor-1}^2-c_{\lfloor x_{\ast}\rfloor-1}+\phi(\lfloor x_{\ast}\rfloor)=\\
&=\frac{\phi(\lfloor x_{\ast}\rfloor)}{c_{\lfloor x_{\ast}\rfloor}^2}(c_{\lfloor x_{\ast}\rfloor}^2-c_{\lfloor x_{\ast}\rfloor}+\phi(\lfloor x_{\ast}\rfloor))>\\
&>0,
\end{align*}
and therefore \eqref{eq:cestindsecond} is valid for $n=0$. If, however, $\lfloor x_{\ast}\rfloor<x_0$, then $\phi(\lfloor x_{\ast}\rfloor)>1/4$ and therefore
\begin{align*}
(c_{\lfloor x_{\ast}\rfloor}-c_{\lfloor x_{\ast}\rfloor-1})(1-c_{\lfloor x_{\ast}\rfloor-1})&=c_{\lfloor x_{\ast}\rfloor-1}^2-c_{\lfloor x_{\ast}\rfloor-1}+\phi(\lfloor x_{\ast}\rfloor)>\\
&>c_{\lfloor x_{\ast}\rfloor-1}^2-c_{\lfloor x_{\ast}\rfloor-1}+\frac{1}{4}=\\
&=\left(c_{\lfloor x_{\ast}\rfloor-1}-\frac{1}{2}\right)^2\geq\\
&\geq0,
\end{align*}
which also yields that \eqref{eq:cestindsecond} is valid for $n=0$. Now let $n\in\{0,\ldots,\lfloor x_{\ast}\rfloor-1\}$ be arbitrary but fixed and assume that \eqref{eq:cestindsecond} is true for $n$ and that $n+1\in\{0,\ldots,\lfloor x_{\ast}\rfloor-1\}$. Then
\begin{equation*}
c_{\lfloor x_{\ast}\rfloor-n-2}=1-\frac{\phi(\lfloor x_{\ast}\rfloor-n-1)}{c_{\lfloor x_{\ast}\rfloor-n-1}}<1-\frac{\phi(\lfloor x_{\ast}\rfloor-n)}{c_{\lfloor x_{\ast}\rfloor-n-1}}<1-\frac{\phi(\lfloor x_{\ast}\rfloor-n)}{c_{\lfloor x_{\ast}\rfloor-n}}=c_{\lfloor x_{\ast}\rfloor-n-1},
\end{equation*}
so \eqref{eq:cestindsecond} is also valid for $n+1$.\\

Combining Step 2 and Step 3, we obtain that $(c_n)_{n\in\mathbb{N}}$ is strictly increasing.
\end{proof}

\bibliography{bibliographyassocpollaczek}
\bibliographystyle{amsplain}

\end{document}